\documentclass[11pt]{amsart}

\usepackage{amsmath, amsthm, amssymb}
\usepackage{amsfonts}
\usepackage{mathtools}

\usepackage[cal=boondox]{mathalfa}
\usepackage{bbm}
\usepackage{mathrsfs}
\usepackage{stmaryrd}
\usepackage{esint}
\usepackage[T1]{fontenc}
\usepackage{lmodern}

\usepackage{graphicx}
\usepackage{tikz}
\usepackage{tikz-cd}
\usetikzlibrary{arrows}
\usepackage[all,cmtip]{xy}
\usepackage{amscd}
\usepackage{picinpar}

\usepackage{enumerate}
\usepackage{enumitem}
\usepackage{comment}
\usepackage{array,booktabs,longtable,ragged2e}
\usepackage{extarrows}
\usepackage{verbatim}
\usepackage{calc}
\usepackage{xcolor}

\usepackage{hyperref}
\usetikzlibrary{calc}

\numberwithin{equation}{section}

\theoremstyle{plain}
\newtheorem{thm}{Theorem}
\numberwithin{thm}{subsection}
\newtheorem{cor}[thm]{Corollary}
\newtheorem{lem}[thm]{Lemma}
\newtheorem{prop}[thm]{Proposition}
\newtheorem{propdef}[thm]{Proposition/Definition}

\newtheorem{mainthm}{Theorem}

\theoremstyle{definition}
\newtheorem{defi}[thm]{Definition}

\newtheorem{example}[thm]{Example}

\theoremstyle{remark}
\newtheorem{rem}[thm]{Remark}

\newcolumntype{L}[1]{>{\RaggedRight\arraybackslash}p{#1}}
\DeclareMathOperator{\PGL}{PGL}

\DeclareMathOperator{\Pic}{Pic}

\DeclareMathOperator{\coh}{H}

\DeclareMathOperator{\Stab}{Stab}

\DeclareMathOperator{\Sing}{Sing}

\DeclareMathOperator{\Crit}{Crit}
\DeclareMathOperator{\Zar}{Zar}

\def\bC{{\mathbb{C}}}
\def\bD{{\mathbb{D}}}

\def\bZ{{\mathbb{Z}}}

\def\bP{{\mathbb{P}}}
\def\bN{{\mathbb{N}}}

\def\bG{{\mathbb{G}}}

\def\cO{{\mathcal{O}}}

\title{Postcritically finite endomorphisms I: Kobayashi hyperbolicity and tautness}

\author{Ruiran Sun}
 \address{School of Mathematical Sciences, Xiamen University, Xiamen 361005, China}
\email{ruiransun@xmu.edu.cn}
\author{Shengyuan Zhao}
\address{Universit\'e Paul Sabatier, Institut de Math\'ematiques de Toulouse, 118, route de Narbonne, F-31062 Toulouse, France}
\email{shengyuan.zhao@math.univ-toulouse.fr}

\begin{document}

\begin{abstract}
Let $f:\bP^N\to\bP^N$ be a postcritically finite endomorphism of degree at least two. Combining Yamanoi's hyperbolicity results with the dynamical property of $f$, we prove that the complement of its postcritical divisor is taut unless $f$ is a monomial power map or admits a nontrivial equivariant rational fibration with a polarized divisorially postcritically finite base map. We establish analogous alternatives for polarized PCF endomorphisms of rationally connected projective manifolds and, for Kobayashi hyperbolicity, on normal rationally connected projective varieties. In dimension two, the non-taut case can be described, up to iteration and rational semiconjugacy, by monomial power maps and skew products with monomial fiber maps.
\end{abstract}

\maketitle

\tableofcontents

\section{Introduction}

Let $N\geq 1$. A postcritically finite (PCF) endomorphism of $\bP^N$ is an endomorphism whose critical divisor has an algbraic forward orbit (this forward orbit is called the PCF divisor). A complex manifold $Y$ is called taut if every sequence of holomorphic maps $\mathbb{D}\to Y$ has a subsequence that either converges uniformly on compact subsets or is compactly divergent (see \cite{Wu}). A complex manifold is Kobayashi hyperbolic if, roughly speaking, holomorphic disks in it are sufficiently constrained that they cannot stretch arbitrarily in any complex direction (see \cite{Kobayashi1998}). These two notions of hyperbolicity are closely related; every taut complex manifold is Kobayashi hyperbolic. In this paper we study the hyperbolicity of the complement of the PCF divisor of a PCF endomorphism. Our motivation comes, on one hand, from the work of Fornaess--Sibony~\cite{FornaessSibonyKobayashi} where Kobayashi hyperbolicity of the PCF complement is used in holomorphic dynamics of PCF maps, on the other hand, from the construction by Koch~\cite{Koch} of interesting examples of PCF maps satisfying the Kobayashi hyperbolicity. 

In complex dimension one, postcritically finite rational maps form a classical and extensively studied part of holomorphic dynamics. From the viewpoint considered in our paper, however, the one-dimensional case is trivial because the complement of the postcritical set in $\bP^1$ is taut when the postcritical set contains at least three points; we refer to \cite{DKMpcfset} for finer information on the configuration of PCF divisors in dimension one. Our focus will therefore be on higher dimensions. 

Let $N\geq 1$. Let $f:\bP^N\to\bP^N$ be an endomorphism of degree at least $2$. Its critical locus is
\[
\Crit(f):=\{x\in\bP^N \mid \operatorname{rank}(d_xf)<N\}.
\]
Its postcritical set is
\[
P_f:=\bigcup_{n\geq 1} f^n(\Crit(f)).
\]
We say that $f$ is \emph{postcritically finite} (PCF) if $P_f$ is an algebraic subset of $\bP^N$. In this case $P_f$ is a divisor because $\Crit(f)$ is a divisor. By abuse of notation we denote by $P_f$ both the postcritical set and its associated reduced effective divisor.
Higher dimensional examples of PCF endomorphisms include certain monomial maps (see Example~\ref{ex:monomial-pcf} below), Lattès maps (see Section~\ref{sec:lattes-maps} and the references therein), maps induced by one dimensional PCF maps via operations such as fibration or symmetric product. There aren't many other known examples of higher dimensional PCF endomorphisms. There are the above mentionned examples of Fornaess--Sibony and Koch~\cite{FornaessSibonyCF,FornaessSibonyKobayashi,Koch} for which Kobayashi hyperbolicity of PCF divisor complement is known (One may also deduce tautness by a theorem of Kiernan~\cite{Kiernan}). A few other examples can be found in \cite{DabijaJonsson,Crass2007Klein,Crass2011Valentiner}. We refer also to \cite{Ueda,Jonsson,RongFatou,Astorg,Ji,Van1,Van2,GTV} for dynamics of PCF endomorphisms.

\begin{example}\label{ex:monomial-pcf}
The following map is a PCF endomorphism:
\[
f:[x_0:\cdots:x_N]\mapsto [\alpha_0 x_{\sigma(0)}^{d}:\cdots:\alpha_N x_{\sigma(N)}^d]
\]
with \(\alpha_0,\cdots,\alpha_N\in \bC^*\) and \(\sigma\in \mathfrak{S}_{N+1}\). The PCF divisor is the union of the $N+1$ coordinate hyperplanes. 
\end{example}

Our first main result relates tautness of the postcritical complement to the algebraic structure of the map. We first recall the terminology. An endomorphism $g:Y\to Y$ is \emph{polarized} if $g^*H\sim qH$ for an ample Cartier divisor $H$ and an integer $q\geq 2$. On a normal variety, it is \emph{divisorially PCF} if the forward images of its ramification divisor form a finite union of prime divisors. See Definition~\ref{def:divisorial-PCF-definition}.
\begin{mainthm}\label{mainthm:Projective-space-pcf}
Let \(f:\bP^N\to\bP^N\) be a PCF endomorphism with PCF set \(P_f\). If \(\bP^N\setminus P_f\) is not taut, then either \(f\) is a monomial power map in suitable homogeneous coordinates (see Example~\ref{ex:monomial-pcf}), or there exist a normal rationally connected projective variety \(Y\) of positive dimension strictly less than \(N\), a dominant rational map with connected fibers \(\pi:\bP^N\dashrightarrow Y\), and a polarized divisorially PCF endomorphism \(g:Y\to Y\) such that \(g\circ\pi=\pi\circ f\).
\end{mainthm}

In fact, Theorem~\ref{mainthm:Projective-space-pcf} is obtained as a special case of the following:
\begin{mainthm}\label{mainthm:rationally-connected}
Let \(f\) be a divisorially PCF polarized endomorphism of a rationally connected complex projective manifold \(X\) of dimension \(N\) with PCF divisor \(P_f\). If \(X\setminus P_f\) is not taut, then one of the following holds: 
\begin{enumerate}
\item there exists a reduced divisor \(D_\infty\subset X\) containing the ramification locus of \(f\) such that \(f^{-1}(D_\infty)=D_\infty\), the pair \((X,D_\infty)\) is toric, and on the torus \(T=X\setminus D_\infty\) the map \(f\) is a translation composed with an isogeny of \(T\);
\item there exist a rationally connected normal projective variety \(Y\) of positive dimension strictly less than \(N\), a dominant rational map with connected fibers
\(\pi:X\dashrightarrow Y\), and a polarized divisorially PCF endomorphism \(g:Y\to Y\) such that \(g\circ\pi=\pi\circ f\).
\end{enumerate}
\end{mainthm}

Note that a PCF endomorphism preserving a fibration may still have a Kobayashi hyperbolic PCF complement, so the trichotomy in Theorem~\ref{mainthm:Projective-space-pcf} is not exclusive. In dimension two we have a more precise trichotomy:
\begin{mainthm}\label{mainthm:P2}
Let \(f:\bP^2\to\bP^2\) be a PCF endomorphism with PCF set \(P_f\). If \(\bP^2\setminus P_f\) is not taut, then there exists \(k\in \bN^*\) such that the iterate \(F:=f^k\), of degree \(D\), satisfies one of the following:
\begin{enumerate}
\item \(F\) is a monomial power map.
\item There exists a rational surface \(S\), a dominant rational map \(\tau:S\dashrightarrow \bP^2\), a dominant rational map \(\rho:S\dashrightarrow \bP^1\), a rational self-map \(G:S\dashrightarrow S\), a PCF map \(h:\bP^1\to \bP^1\) such that
\[
\tau \circ G=F\circ \tau, \quad \rho \circ G=h\circ \rho,
\]
and, in coordinates trivializing \(\rho\) over a Zariski-open subset of \(\bP^1\), 
\[
G(t,u)=\bigl(h(t), a(t)u^D \bigr),\quad
a(t)\in \bC(t)^*,\; u\in \bC^*.
\]
\end{enumerate}
\end{mainthm}

As Theorem~\ref{mainthm:Projective-space-pcf} suggest, PCF endomorphisms of varieties other than projective spaces may, a priori, show up as reduction factors even if one studies only PCF endomorphisms of projective spaces. Note that examples of divisorially PCF polarized endomorphisms of non-smooth non-rational normal Fano varieties were constructed by Koll\'ar--Xu~\cite{KollarXu}. 

For PCF endomorphism on such normal varieties, we can still prove a Kobayashi hyperbolicity alternative:

\begin{mainthm}\label{mainthm:rationally-connected_Kob}
Let \(f\) be a divisorially PCF polarized endomorphism of a rationally connected normal complex projective variety \(X\) of dimension \(N\) with PCF divisor \(P_f\) and singular locus \(\Sing(X)\). If \(X\setminus (P_f\cup \Sing(X))\) is not Kobayashi hyperbolic, then one of the following holds: 
\begin{enumerate}
\item there exists a divisor \(D_\infty\subset X\) containing the codimension one part of the ramification locus of \(f\) such that \(f^{-1}(D_\infty)=D_\infty\); 
\item there exist a rationally connected normal projective variety \(Y\) of positive dimension strictly less than \(N\), a dominant rational map with connected fibers \(\pi:X\dashrightarrow Y\), and a polarized divisorially PCF endomorphism \(g:Y\to Y\) such that \(g\circ\pi=\pi\circ f\).
\end{enumerate}
\end{mainthm}

In the final section we also study the explicit PCF endomorphism
\[
F[X:Y:Z]=[X^2-2YZ:Y^2-2XZ:Z^2]
\]
of $\bP^2$. Its postcritical complement $\bP^2\setminus P_F$ is taut, but is not hyperbolically embedded into $\bP^2$. Moreover, no positive iterate of $F$ preserves a rational pencil. See Proposition~\ref{hyper-emb-counterexample} for details.

\subsection{Strategy and plan}\label{subsec:plan}
Our proof combines quasi-Albanese morphisms with the backward dynamics of $f$. For $X$ and $f$ as in Theorem~\ref{mainthm:rationally-connected_Kob}, removing the successive inverse images of $P_f\cup\Sing(X)$ gives a \emph{tower of finite \'etale coverings} of the smooth postcritical complement $U:= X \setminus \left(P_f\cup\Sing(X)\right)$. If only finitely many prime boundary divisors occur, their union is totally invariant and contains the ramification divisor. In the smooth case, Meng--Zhong's theorem identifies the resulting pair as toric, yielding the first alternative.

Otherwise, the growing boundary supplies increasingly many units in the function field. The resulting unit maps have an intrinsic interpretation as \emph{quasi-Albanese morphisms}. Note that rational connectedness removes the abelian part, and thus the quasi-Albanese varieties of $U$ are algebraic tori. Their images connect the dynamics to Yamanoi's celebrated hyperbolicity theorem~\cite{Yamanoi2026}. Taking account of their translation symmetries and comparing these maps under iteration, we obtain either a nontrivial equivariant rational fibration or local lower bounds for the Kobayashi--Royden metric outside a proper algebraic subset. In the fibration case, polarization and divisorial postcritical finiteness descend to the base.

The dynamics then gives a \emph{dense-or-nothing} principle for metric degeneracy. Every point of the smooth postcritical complement has a Zariski-dense backward orbit, and failure of a local positive lower bound propagates backwards through the \'etale tower. The degeneracy locus is therefore either empty or Zariski dense (cf. Proposition~\ref{prop:dynamical-amplification}). Yamanoi's estimates confine it to a proper algebraic subset, forcing it to be empty. This proves the Kobayashi hyperbolicity alternative.

To strengthen this hyperbolicity conclusion to tautness when $X$ is smooth, we use a \emph{bubbling argument}. A sequence of discs witnessing non-tautness can be lifted through the tower with centers outside the algebraic exceptional set. Hurwitz's theorem preserves nonnormality in $X$. By Yamanoi's theorem, the images under the associated torus map have a subsequence converging in a compactification. An algebraic comparison then gives uniform area bounds on smaller discs. A second lift through a high iterate again has bounded area. Then one can use the Brody reparametrization to produce a nonconstant entire curve with \emph{finite area}, which extends to a rational curve in $X$. After composition with that iterate, this rational curve is constrained by the first area bound. Polarization forces its area to exceed that bound when the second lift is sufficiently deep, giving the contradiction.

The paper begins with a toy example of Latt\`es maps in Section~\ref{sec:lattes-maps}. Sections~\ref{sec:PCF-rc-chow} and~\ref{sec:semiabelian-hyperbolicity} develop the equivariant reductions and quasi-Albanese construction. We prove the hyperbolicity and tautness alternatives in Section~\ref{sec:main-proof-hyperbolicity}, then deduce the projective-space statement in Section~\ref{sec:ramification-and-toric}. The final section uses the classification of invariant pencils to prove Theorem~\ref{mainthm:P2} and studies the aforementioned example.

\section{Toy example: Lattès maps}\label{sec:lattes-maps}

\subsection{Preliminaries}\label{subsec:lattes-definition}

We say that an endomorphism $f:\bP^N\to\bP^N$ is a \emph{Lattès map} if there exist a complex abelian variety $A$ of dimension $N$, an affine endomorphism
\[
L=t_a\circ\varphi:A\longrightarrow A,
\]
where $\varphi:A\to A$ is an isogeny and $t_a$ is a translation, and a finite surjective morphism $\pi:A\to\bP^N$ such that $f\circ\pi=\pi\circ L$.

We now prove that such a Lattès map $f$ is PCF. Let $B_\pi:=\pi(\Crit(\pi))$ be the branch locus of $\pi$. It is a divisor. We claim that
\[
f^n(\Crit(f))\subset B_\pi
\quad\text{for every }n\geq1.
\]
Let $x\in\Crit(f)$ and choose $y\in A$ with $\pi(y)=x$. Since $d_xf$ is singular, the differential $d_x(f^n)$ is singular for every $n\geq 1$. From the identity \(f^n\circ\pi=\pi\circ L^n\) we obtain
\begin{equation}\label{eq:lattes-differential}
d_{L^n(y)}\pi\circ d_yL^n
=
d_x(f^n)\circ d_y\pi.
\end{equation}
The right-hand side is singular. Since $L$ is étale, $d_yL^n$ is invertible, and hence $d_{L^n(y)}\pi$ is singular. Thus $L^n(y)\in\Crit(\pi)$, and therefore
\[
f^n(x)=\pi(L^n(y))\in B_\pi.
\]
This proves the claim.
The postcritical set thus satisfies \(P_f\subset B_\pi\). In particular $P_f$ is contained in a proper Zariski-closed subset of $\bP^N$, so $f$ is PCF.

Actually we have $P_f=B_\pi$. The reason is as follows. Let
$\Crit(\pi)=D_1\cup\ldots\cup D_r$ be its irreducible decomposition. Let $H:=\pi^*\cO_{\bP^N}(1)$ and $d:=\deg f$. Since $\pi$ is finite, $H$ is ample and $H^N=\deg(\pi)>0$. Note that $L^*H\simeq H^{\otimes d}$.
By the projection formula, we have
$d^{nN}H^N=((L^n)^*H)^N=\deg(L^n)H^N$, so $\deg(L^n)=d^{nN}$. Again by the projection formula we get
\[
d^{n(N-1)}(L^n)^*D_i\cdot H^{N-1}=(L^n)^*D_i\cdot\bigl((L^n)^*H\bigr)^{N-1}=d^{nN}D_i\cdot H^{N-1},
\]
so $(L^n)^*D_i\cdot H^{N-1}=d^nD_i\cdot H^{N-1}$. Since $(L^n)^*D_i$ is reduced because $L$ is \'etale, if all its components lied in $\Crit(\pi)$, then we would have
\[
d^nD_i\cdot H^{N-1}\leq\Crit(\pi)_{\mathrm{red}}\cdot H^{N-1},
\]
impossible for $n$ large. Fix such an $n$ and a component $E$ of $(L^n)^*D_i$ with
$E\not\subset\Crit(\pi)$, so $L^n(E)=D_i$. Let $j<n$ be the minimal integer such that $L^{j+1}(E)\subset\Crit(\pi)$, and let $E':=L^j(E)$. For general $y\in E'$, taking $n=1$
in \eqref{eq:lattes-differential}, remarking that the maps $d_yL$ and $d_y\pi$ are invertible while $d_{L(y)}\pi$ is singular, we deduce that $d_{\pi(y)}f$ is singular. Hence
$\pi(E')\subset\Crit(f)$ and $\pi(D_i)=f^{n-j}(\pi(E'))\subset P_f$. Therefore $B_\pi\subset P_f$.

Let us summarize:
\begin{prop}\label{prop:lattes-pcf}
Every Lattès map $f:\bP^N\to\bP^N$ is PCF, with PCF divisor equal to the branching locus of the associated cover of $\bP^N$ by the abelian variety.
\end{prop}

\subsection{Kobayashi hyperbolicity of Lattès PCF complements}\label{subsec:Lattes-pcf-complements}
We give an illustration, in the case of Lattès maps, of how Yamanoi's work is relevant to our study of PCF endomorphisms. The following theorem of Yamanoi is a deep result relying on many prior theories. 

\begin{thm}[{Yamanoi~\cite[Cor.~1.2]{YamanoiAbelian}}]\label{thm:Yamanoi-abelian}
If $A$ is a complex abelian
variety and $D$ is a reduced ample divisor on $A$, then $A\setminus  D$ is Kobayashi hyperbolic.
\end{thm}

\begin{cor}\label{cor:lattes-kobayashi-hyperbolic}
For any Lattès map, the complement of its PCF divisor is Kobayashi hyperbolic. 
\end{cor}
\begin{proof}
Let $f:\bP^N\to\bP^N$ be a Lattès map as in \S\ref{subsec:lattes-definition}. 
Since $P_f$ is a nonzero effective divisor on $\bP^N$, it is ample. As $\pi$ is a finite morphism, the pullback $\pi^*P_f$ is again ample, so is its underlying reduced divisor. Yamanoi's theorem shows that $A\setminus\pi^{-1}(P_f)$ is Kobayashi hyperbolic. Since $P_f=B_\pi$ by Prop.~\ref{prop:lattes-pcf}, the restriction
\[
\pi:A\setminus\pi^{-1}(P_f)\to\bP^N\setminus P_f
\]
is a finite unramified covering. Kobayashi hyperbolicity descends under
finite unramified coverings (see \cite[Th.~3.2.8]{Kobayashi1998}). Hence $\bP^N\setminus P_f$ is Kobayashi hyperbolic.
\end{proof}

\section{Postcritically finiteness on normal varieties and Chow reduction}\label{sec:PCF-rc-chow}

Throughout this section, all varieties are irreducible complex algebraic varieties.  For a normal variety \(X\), a prime divisor means an irreducible closed subvariety of codimension
one.  We write \(\bC(X)\) for the function field of \(X\). 

\subsection{Chow reduction and equivariant regularization}\label{subsec:chow-reduction}

Let
\[
 \pi:X\dashrightarrow Y
\]
be a dominant rational map between normal projective varieties.  Let
\(\Gamma_\pi\) be the normalization of the closure of its graph in \(X\times Y\), with the natural projection morphisms
\[
 p_\pi:\Gamma_\pi\longrightarrow X,
 \qquad
 q_\pi:\Gamma_\pi\longrightarrow Y.
\]
Note that \(p_\pi\) is birational and \(q_\pi\) is surjective.  Following
\cite{MengZhang}, we say that \(\pi\) has \emph{connected fibers} if \(q_\pi\) has connected fibers.  By Stein factorization, this is equivalent to the condition that \(\bC(Y)\) is relatively algebraically closed in \(\bC(X)\).

A dominant morphism \(q:Z\to Y\) between irreducible varieties is called
\emph{equidimensional} if every irreducible component of every fiber of \(q\) has dimension \(\dim Z-\dim Y\). 

\begin{lem}\label{lem:equidimensional-normalization}
Let $q:Z\to Y$ be a dominant morphism of irreducible varieties, and let $\nu:Z^\nu\to Z$ be the normalization. Then $q$ is equidimensional if and only if $q\circ\nu:Z^\nu\to Y$ is equidimensional.
\end{lem}

\begin{proof}
 Put $d=\dim Z-\dim Y$. For a dominant morphism of irreducible varieties, one can use the upper semicontinuity of local fiber dimension to show that every component of every nonempty fiber has dimension at least $d$. On the other hand, finite surjections preserve the total dimension of each fiber. Therefore, every fiber is pure of dimension $d$ before normalization if and only if this holds afterwards.
\end{proof}

Now let us introduce the notion of Chow reduction due to Nakayama~\cite[Prop.~4.14, Def.~4.15]{Nakayama}; see also \cite[Prop.~7.1]{MengZhang}.
\begin{propdef}[Nakayama~\cite{Nakayama}]\label{prop:chow-reduction-def}
Let \(\pi:X\dashrightarrow Y\) be a dominant rational map from a projective variety
\(X\) to a normal projective variety \(Y\).  There exist a normal projective variety
\(T\) and a birational map
\[
 \mu:Y\dashrightarrow T
\]
such that \(\mu\circ\pi\) is equidimensional.  Moreover, this model is minimal in the following sense. If \(\mu':Y\dashrightarrow T'\) is another birational model, with \(T'\) normal projective and \(\mu'\circ\pi\) equidimensional, then there is a birational morphism \(\nu:T'\to T\) such that \(\mu=\nu\circ\mu'\). The rational map
\[
 \mu\circ\pi:X\dashrightarrow T
\]
is called the \emph{Chow reduction} of \(\pi\); it is unique up to isomorphism.
\end{propdef}

We shall use the following reformulation of Nakayama's result:
\begin{thm}[{Nakayama~\cite[Th.~4.19]{Nakayama}}, {Meng--Zhang~\cite[Th.~7.2]{MengZhang}}]\label{thm:chow-reduction-endomorphism}
Let \(\pi:X\dashrightarrow Y\) be a dominant rational map between normal projective varieties with connected fibers. Let \(f:X\to X\) be a surjective endomorphism and let \(g:Y\dashrightarrow Y\) 
be a dominant rational self-map such that \(g\circ\pi=\pi\circ f\). 
If \(\pi\) is its own Chow reduction, then \(g\) is a surjective endomorphism of
\(Y\).
\end{thm}

We shall also use the following related lemmas. 
\begin{lem}\label{lem:invariant-subfield-chow}
Let \(X\) be a normal projective variety, let \(f:X\to X\) be a surjective endomorphism, and let \(\bC\subsetneq L\subsetneq \bC(X)\) be a finitely generated subfield over \(\bC\) which is relatively algebraically closed in \(\bC(X)\) and satisfies \(f^*L\subset L\). 
Then there exist a normal projective variety \(Y\), a dominant rational map with connected fibers \(\pi:X\dashrightarrow Y\) and a surjective endomorphism \(g:Y\to Y\) 
such that
\[
 g\circ\pi=\pi\circ f,\quad \bC(Y)=L,
\]
and \(\pi\) is its own Chow reduction.
\end{lem}

\begin{proof}
Fix a normal projective variety \(Y_0\) with function field \(L\). The inclusion \(L\subset \bC(X)\) induces a dominant rational map \(\pi_0:X\dashrightarrow Y_0\). 
Since \(L\) is relatively algebraically closed in \(\bC(X)\), the map \(\pi_0\) has connected fibers.  The inclusion \(f^*L\subset L\) induces a dominant rational self-map \(g_0:Y_0\dashrightarrow Y_0\) satisfying \(g_0\circ\pi_0=\pi_0\circ f\). 

Let \(\mu:Y_0\dashrightarrow Y\) be the Chow reduction of \(\pi_0\), and let
\[
 \pi:=\mu\circ\pi_0,
 \quad
 g:=\mu\circ g_0\circ\mu^{-1}.
\]
Then \(\bC(Y)=L\), so \(\pi\) still has connected fibers. Theorem~\ref{thm:chow-reduction-endomorphism} therefore shows that
\(g\) is a surjective endomorphism.
\end{proof}

Recall that an endomorphism \(f:X\to X\) is called \emph{\(q\)-polarized} if there are an integer \(q\geq 2\) and an ample Cartier divisor \(H\) on \(X\) such that \(f^*H\sim qH\). The following is an immediate application of a theorem of Meng--Zhang:
\begin{lem}[{Meng--Zhang~\cite[Theorem~1.3]{MengZhangPolarized}}]\label{lem:polarization-descent}
In the setting of Lemma~\ref{lem:invariant-subfield-chow}, suppose that \(f\) is \(q\)-polarized. Then \(g\) is \(q\)-polarized.
\end{lem}

\subsection{Divisorial postcritical finiteness on a normal variety}\label{subsec:divisor-PCF-normal-variety}

Let \(f:X\to X\) be a finite surjective endomorphism of a normal projective variety. If \(E\subset X\) is a prime Weil divisor and \(D=f(E)\), then the local rings at the generic points of \(E\) and \(D\) are discrete valuation rings, and we denote by \(e_f(E)\) the ramification index of the induced extension of discrete valuations. We define the \emph{ramification divisor} as
\[
 R_f^{(1)}:=\sum_E\bigl(e_f(E)-1\bigr)E.
\]
It is a Weil divisor encoding the codimension-one part of the ramification. We say that the morphism \(f\) is \emph{quasi-étale} if it is étale at every codimension-one point, equivalently if \(R_f^{(1)}=0\). 

\begin{defi}\label{def:divisorial-PCF-definition}
The \emph{divisorial postcritical set} of \(f\) is
\[
 P_f^{(1)}:=\bigcup_{n\geq 1} f^n\bigl(\operatorname{Supp}R_f^{(1)}\bigr).
\]
We say that \(f\) is \emph{divisorially postcritically finite}, or \emph{divisorially
PCF}, if \(P_f^{(1)}\) is a finite union of prime divisors.
\end{defi}
\begin{rem}
When \(X\) is smooth, this coincideds with the usual definition of postcritical finiteness. 
\end{rem}

By construction \(f\) is quasi-étale outside \(R_f^{(1)}\). To obtain open subsets on which \(f\) is genuinely étale, we adjoin the singular locus as follows. We define, for \(n\in \bN\), the subsets
\begin{equation}\label{eq:definition-adjoin-locus}
\Delta(f)=\Delta_0(f):=\operatorname{Sing}(X)\cup P_f^{(1)},
\quad \Delta_n(f):=f^{-n}(\Delta(f)),
\quad U_n(f):=X\setminus\Delta_n(f).
\end{equation}
We denote by \(D_n(f)\) the reduced Weil divisor of \(X\) whose prime components are precisely the codimension-one irreducible components of \(\Delta_n(f)\).

\begin{lem}\label{lem:singular-boundary-forward}
Let \(f:X\to X\) be a finite surjective endomorphism of a normal projective variety.
Then
\[
 f\bigl(\operatorname{Sing}X\bigr)
 \subset
 \operatorname{Sing}X\cup f\bigl(\operatorname{Supp}R_f^{(1)}\bigr).
\]
Consequently, if \(f\) is divisorially PCF, then
\[
 f\bigl(\Delta(f)\bigr)\subset\Delta(f).
\]
\end{lem}

\begin{proof}
Let \(x\in\operatorname{Sing}X\) and \(y=f(x)\). Suppose that \(y\) is smooth and
that \(y\notin f(\operatorname{Supp}R_f^{(1)})\). Since \(f\) is finite, the latter set is Zariski-closed. Hence there is a Zariski-open neighborhood \(V\) of \(y\), contained in the regular locus \(X_{\mathrm{reg}}\), such that no codimension-one ramification locus meets \(f^{-1}(V)\). Denote 
\[
g:=f|_{f^{-1}(V)}:f^{-1}(V)\to V.
\]
The morphism \(g\) is finite, its source is normal, and its target is regular. Since \(g\) is quasi-\'etale, it is unramified at every
codimension-one point of \(f^{-1}(V)\). Hence the Zariski--Nagata purity theorem \cite[Th.~X.3.1]{SGA1} implies that \(g\) is \'etale. 
Since \(V\) is regular and
regularity is preserved under \'etale morphisms, \(f^{-1}(V)\) is regular,
contradicting \(x\in\operatorname{Sing}X\). Thus
\[
f\bigl(\operatorname{Sing}X\bigr)
\subset
\operatorname{Sing}X\cup
f\bigl(\operatorname{Supp}R_f^{(1)}\bigr).
\]

Assume that \(f\) is divisorially PCF. Since \(P_f^{(1)}\) contains \(f(\operatorname{Supp}R_f^{(1)})\) and is forward invariant, we obtain
\[
f\bigl(\Delta(f)\bigr)
=
f\bigl(\operatorname{Sing}X\cup P_f^{(1)}\bigr)
\subset
\operatorname{Sing}X\cup P_f^{(1)}
=
\Delta(f).
\]
\end{proof}

\begin{lem}\label{lem:etale-tower}
Assume that \(f\) is divisorially PCF.  Then for every \(n\) we have
\[
 \Delta_n(f)\subset\Delta_{n+1}(f),
 \quad
 U_{n+1}(f)\subset U_n(f),
\]
and the restriction \(f:U_{n+1}\to U_n\) is finite étale. Consequently, \(f^k:U_{n+k}\to U_n\) is finite étale for every \(n,k\geq 0\).
\end{lem}

\begin{proof}
Lemma~\ref{lem:singular-boundary-forward} implies that \(\Delta_f\subset f^{-1}(\Delta_f)\). Pulling back yields
\(\Delta_n(f)\subset\Delta_{n+1}(f)\). We have also \(f^{-1}(U_n)=U_{n+1}\). Let \(x\in U_{n+1}\) and \(y=f(x)\). Since
\(\operatorname{Sing}X\subset\Delta\subset\Delta_n\), both \(x\) and \(y\) are smooth. The smoothness implies that if \(f\) were ramified at \(x\) then the point \(x\) would lie in
\(\operatorname{Supp}R_f^{(1)}\). Hence
\[
 y\in f\bigl(\operatorname{Supp}R_f^{(1)}\bigr)
 \subset P_f^{(1)}\subset\Delta(f)\subset\Delta_n(f),
\]
contrary to \(y\in U_n(f)\). Thus the restriction morphism is étale.
\end{proof}

In terms of the Weil divisors \(D_n(f)\), we can also reformulate some information of Lemma~\ref{lem:etale-tower} in the following way. 
\begin{cor}\label{cor:divisorial-tower}
Assume that \(f\) is divisorially PCF.  Then
\[
 D_n(f)\subset D_{n+1}(f),
 \quad
 D_{n+1}(f)=\bigl(f^{*}D_n(f)\bigr)_{\mathrm{red}},
 \quad 
\operatorname{Supp}R_f^{(1)}\subset D_1(f).
\]
In particular, if the numbers of irreducible components of the \(D_n\) are bounded,
then the sequence stabilizes to a reduced divisor \(D_\infty\) satisfying
\[
 f^{-1}(D_\infty)=D_\infty,
 \qquad
 \operatorname{Supp}R_f^{(1)}\subset D_\infty.
\]
\end{cor}

\subsection{Chow reduction of a divisorial PCF map}\label{subsec:chow-of-pcf}

\begin{lem}\label{lem:multiplicity-one-open}
Let \(q:Z\to Y\) be a dominant projective morphism of normal complex varieties, and
assume that \(\bC(Y)\) is relatively algebraically closed in \(\bC(Z)\). Then there is a nonempty Zariski-open subset \(Y^\circ\subset Y\) with the following property. If \(D\subset Y\) is a prime divisor whose generic point belongs to \(Y^\circ\), and if \(E\subset Z\) is a prime divisor dominating \(D\), then, for a uniformizer \(t\) of the discrete valuation ring \(\mathcal O_{Y,\eta_D}\), we have \(\operatorname{ord}_E(q^*t)=1\).
\end{lem}

\begin{proof}
The relative algebraic closedness of \(\bC(Y)\) in \(\bC(Z)\) implies that the generic fiber of \(q\) is geometrically integral. By the generic property of flatness~\cite[Th.~6.9.1]{EGAIV} and openness of geometric reducedness~\cite[Th.~12.2.4]{EGAIV} we obtain a nonempty open subset \(Y^\circ\) over which \(q\) is flat with geometrically reduced fibers. For \(D\) and \(E\) as in the statement, \(\operatorname{ord}_E(q^*t)\) is the multiplicity of the component determined by \(E\) in the fiber over the generic point of \(D\). Since that fiber is reduced, this multiplicity is one.
\end{proof}

\begin{prop}\label{prop:pcf-descent}
Let \(X\) and \(Y\) be normal projective varieties.  Let \(f:X\to X, g:Y\to Y\)
be finite surjective endomorphisms and let
\(\pi:X\dashrightarrow Y\) be a dominant rational map with connected fibers such that \(g\circ\pi=\pi\circ f\). Assume that \(\pi\) is its own Chow reduction.  If \(f\) is divisorially PCF, then \(g\) is divisorially PCF.
\end{prop}

\begin{proof}
Let \(\Gamma\) be the normalization of the closure of the graph of \(\pi\), with projections
\[
 p:\Gamma\to X, \quad q:\Gamma\to Y.
\]
Then \(p\) is birational and \(q\) is equidimensional because \(\pi\) is its own Chow reduction. The finite morphism
\(f\times g:X\times Y\to X\times Y\) preserves the closure of the graph of \(\pi\). 
By the universal property of normalization, it therefore induces a dominant endomorphism
\begin{equation}\label{eq:graph-equivariance}
F:\Gamma\to\Gamma, \quad \text{satisfying} \quad p\circ F=f\circ p, \; q\circ F=g\circ q.
\end{equation}

Let us show that \(F\) is a finite morphism. 
Since \(f\times g\) is finite, its restriction
to the closure of the graph is finite. Denoting \(h\) this restriction and \(\nu\) the normalization, we obtain that \(\nu\circ F=h\circ\nu\) is finite. It follows that \(F\) has finite fibers, because \(F^{-1}(z)\subset(\nu\circ F)^{-1}(\nu(z))\) for every
\(z\in\Gamma\). 

We now show that \(F\) is divisorially PCF.  Let \(E\subset\Gamma\) be a ramification
prime divisor of \(F\). We distinguish two cases depending on whether \(E\) is \(p\)-exceptional. If \(E\) is \(p\)-exceptional, then \(F(E)\) is again \(p\)-exceptional. Indeed, \(F(E)\) is a prime divisor because \(F\) is finite, while \(p(F(E))=f(p(E))\) has codimension at least two. Since a birational morphism has only finitely many exceptional prime divisors, the forward orbit of \(E\) is finite. Suppose now that \(E\) is not \(p\)-exceptional. Then \(D:=p(E)\) is a prime
divisor of \(X\). Since \(f(D)\) is again a divisor, \eqref{eq:graph-equivariance} implies that \(F(E)\) is not \(p\)-exceptional.
Since \(p\) is birational and \(X\) and \(Z\) are normal, it is an isomorphism at the generic point of every nonexceptional prime divisor. In particular, \(\operatorname{ord}_E(p^*t)=1\) for a uniformizer \(t\) at the generic point \(\eta_D\). By multiplicativity of the indices, \(p\circ F=f\circ p\) implies \(e_F(E)=e_f(D)\). Thus \(D\subset\operatorname{Supp}R_f^{(1)}\).  Repeating the same argument we deduce that
\(F^j(E)\) is the strict transform of \(f^j(D)\) for every \(j\ge0\). Since \(f\) is
divisorially PCF, only finitely many divisors \(f^j(D)\) occur. Hence the forward orbit of \(E\) is finite, and \(F\) is divisorially PCF.

Now let \(D_0\subset Y\) be a ramification prime divisor of \(g\), and denote
\[
 D_i:=g^i(D_0), \quad i\geq 0.
\]
Suppose for contradiction that the \(D_i\) are pairwise distinct. Since \(Y\) is normal, \(D_0\) is Cartier in a neighborhood of its generic point. Hence there exists a prime divisor \(E_0\subset\Gamma\) dominating \(D_0\). Denote
\[
 E_i:=F^i(E_0), \quad i\geq 0.
\]
Then \(q(E_i)=D_i\) for every \(i\). Let \(t_i\) be a uniformizer at the generic point of \(D_i\) and denote
\[
 m_i:=\operatorname{ord}_{E_i}(q^*t_i),
 \quad
 a_i:=e_g(D_i),
 \quad
 b_i:=e_F(E_i).
\]
Evaluating the identity \(q\circ F=g\circ q\) on the divisorial valuations, we obtain
\begin{equation}\label{eq:ramification-index-relation-prelim}
 b_i m_{i+1}=a_i m_i.
\end{equation}
Recall that \(D_0\) is ramified by choice, so \(a_0\geq 2\).

As \(\pi\) has connected fibers,
\(\bC(Y)\) is relatively algebraically closed in \(\bC(\Gamma)=\bC(X)\). By Lemma~\ref{lem:multiplicity-one-open}, there is a nonempty open subset
\(Y^\circ\subset Y\) such that \(m_i=1\) whenever the generic point of \(D_i\) lies in
\(Y^\circ\).  The complement \(Y\setminus Y^\circ\) contains only finitely many prime
divisors.  Since the \(D_i\) are pairwise distinct by our contradiction hypothesis, it follows that \(m_i=1\) for all sufficiently large \(i\). Multiplying
\eqref{eq:ramification-index-relation-prelim} for \(i=0,\ldots,n-1\) we get
\[
 m_n\prod_{i=0}^{n-1}b_i
 =
 m_0\prod_{i=0}^{n-1}a_i.
\]
For sufficiently large \(n\), we have \(m_n=1\), while the right-hand
side is at least \(2m_0\).  Hence \(b_j>1\) for some \(j\), in other words \(E_j\) is a ramification prime divisor of \(F\). Its forward orbit is finite because we already proved that \(F\) is divisorially PCF. 
Therefore \(D_{j+k}=q(F^k(E_j))\) has only finitely many possibilities, contradicting the assumption that the \(D_i\) are pairwise distinct. Thus every ramification prime divisor of \(g\) has finite forward
orbit, i.e.\ \(g\) is divisorially PCF.
\end{proof}

\subsection{Reduction of PCF maps}\label{subsec:reduction-of-pcf}

Let \(X\) be a normal rationally connected projective variety of dimension \(m\), and let
\(f:X\to X\) be a \(q\)-polarized divisorially PCF endomorphism. Let \(r_n\) be the number of prime components of \(D_n(f)\). Assume that the sequence \((r_n)\) is unbounded. Then by Corollary~\ref{cor:divisorial-tower}, we have
\(\lim r_n=+\infty\). 

We choose a resolution of singularities \(\mu:\widetilde X\to X\) which is an
isomorphism over the regular locus \(X_{\mathrm{reg}}\). Since rational connectedness is a birational invariant by definition, \(\widetilde X\) is rationally connected. Therefore by \cite[Cor.~4.12]{DebarreBook} we have \(\coh^1(\widetilde X,\mathcal O_{\widetilde X})=0\), so \(\operatorname{Pic}^0(\widetilde X)=0\). It follows that \(\Pic(\widetilde X)\) is a finitely generated abelian group. 

Recall that \(\operatorname{Sing}(X)\subset\Delta_n(f)\), so we have
\(U_n(f)\subset X_{\mathrm{reg}}\). Hence \(\mu\) restricts to an isomorphism
\(\widetilde U_n:=\mu^{-1}(U_n(f))\to U_n(f)\). Let \(B_n\) be the reduced
divisor on \(\widetilde X\) formed by the codimension-one irreducible components
of \(\widetilde X\setminus\widetilde U_n\). As \(X\) is normal, its singular locus has codimension \(\geq 2\). Therefore every irreducible component of \(D_n(f)\)
has a strict transform among the components of \(B_n\). We then deduce that the number of irreducible components of \(B_n\) is at least \(r_n\). 

\begin{lem}\label{lem:unit-growth}
For every \(n\in \bN\), the group \(M_n:=\mathcal O(U_n(f))^*/\bC^*\) is a finitely generated
free abelian group. Moreover, \(\lim\operatorname{rank}M_n=+\infty\).
\end{lem}
\begin{proof}
A non vanishing rational function on \(U_n(f)\) is a rational function on \(\widetilde X\) whose divisor is supported on \(B_n\). 
We have an exact sequence
\[
\bC^*\simeq \mathcal O(\widetilde X)^*\rightarrow
\mathcal O(\widetilde U_n)^*
\xrightarrow{\operatorname{divisor}}
\bigoplus_{E\in \operatorname{Irr}(B_n)}\bZ[E]
\rightarrow \Pic(\widetilde X),
\]
where \(\operatorname{Irr}(B_n)\) is the set of irreducible components of \(B_n\) and the third map sends a divisor \(D\) to \(\mathcal O_{\widetilde X}(D)\). Recall that we explained above that \(\Pic(\widetilde X)\) is a finitely generated abelian group. Consequently,
\[
\operatorname{rank}M_n\geq r_n-\operatorname{rank}\Pic(\widetilde X).
\]
Therefore \(M_n\) is isomorphic to a subgroup of a finitely generated free abelian group, hence is itself finitely generated and free. 
The limit \(\lim\operatorname{rank}M_n=+\infty\) follows from \(\lim r_n=+\infty\). 
\end{proof}

In the sequel of this subsection we will denote \(K:=\bC(X)\). For \(n\in \bN\) let \(K_n\subset K\) be the subfield generated over
\(\bC\) by the elements of \(\mathcal O(U_n(f))^*\). Each field \(K_n\) is finitely generated over \(\bC\). The inclusion \(U_{n+1}(f)\subset U_n(f)\) implies that
\(K_n\subset K_{n+1}\). Let \(K_\infty:=\bigcup_{n\geq0}K_n\) and
\(t\) be the transcendence degree of \(K_\infty\) over \(\bC\). By
Lemma~\ref{lem:unit-growth}, some \(U_n(f)\) has a nonconstant nonvanishing function. Hence
\(t\geq 1\). Recall that \(m\) is the dimension of \(X\), so \(m\geq t\). 

\begin{lem}\label{lem:invariant-field}
Assume that \(t<m\). There exists a subfield
\(L\subset\bC(X)\) which is finitely generated over \(\bC\) with transcendence degree \(t\), is relatively algebraically closed in \(\bC(X)\), and satisfies \(f^*L\subset L\). 
\end{lem}

\begin{proof}
The sequence \(\operatorname{trdeg}_{\bC}K_n\) is nondecreasing and has
supremum \(t\). Hence there exists \(n_0\) such that
\(\operatorname{trdeg}_{\bC}K_{n_0}=t\). It follows that
\(K_n/K_{n_0}\) is algebraic for every \(n\geq n_0\).

Let \(L\) be the relative algebraic closure of \(K_{n_0}\) in \(K=\bC(X)\). It is a finite extension of \(K_{n_0}\), and has transcendence degree \(t\). Since \(1\leq t<m\), it is a proper nonconstant subfield of \(K\). 

We have \(f^*K_{n_0}\subset K_{n_0+1}\), while \(K_{n_0+1}/K_{n_0}\) is algebraic.
If \(a\in L\), then \(a\) is algebraic over \(K_{n_0}\), so \(f^*a\) is
algebraic over \(K_{n_0+1}\), thus algebraic over \(K_{n_0}\). Hence \(f^*L\subset L\).
\end{proof}

\begin{prop}\label{prop:pcf-rational-fibration}
Let \(X\) be a normal rationally connected projective variety of dimension
\(m\), and let \(f:X\to X\) be a \(q\)-polarized divisorially PCF
endomorphism. Suppose that \(\operatorname{trdeg}_{\bC}K_\infty<m\). Then there exist a normal rationally connected projective variety \(Y\), with \(0<\dim Y<m\), a dominant rational map with connected fibers
\(\pi:X\dashrightarrow Y\), and a \(q\)-polarized divisorially PCF
endomorphism \(g:Y\to Y\) such that \(g\circ\pi=\pi\circ f\). Moreover, \(\pi\) is its own Chow reduction.
\end{prop}

\begin{proof}
Let \(L\) be the field furnished by Lemma~\ref{lem:invariant-field}. Lemma~\ref{lem:invariant-subfield-chow} then provides a normal projective variety \(Y\), a dominant rational map with connected fibers
\(\pi:X\dashrightarrow Y\), and a surjective endomorphism \(g:Y\to Y\) such
that \(\bC(Y)=L\), \(g\circ\pi=\pi\circ f\), and \(\pi\) is its own Chow
reduction. It follows from the definition of rational connectedness that \(Y\) is rationally connected as the image under a dominant rational map of a rationally connected variety. We have \(\dim Y=t\), so \(0<\dim Y<m\). By Lemma~\ref{lem:polarization-descent}, the map \(g\) is \(q\)-polarized. Proposition~\ref{prop:pcf-descent} then implies that \(g\) is divisorially PCF. 
\end{proof}

\section{Semi-abelian varieties and hyperbolicity}\label{sec:semiabelian-hyperbolicity}
Some algebro-geometric and complex-analytic properties of semi-abelian varieties are crucial for our work. In this section we prepare an introduction to semi-abelian varieties and to some deep hyperbolicity
results that will be used later. 

\subsection{Semi-abelian varieties and compactifications}\label{subsec:semiabelian}

A \emph{semi-abelian variety} is a connected commutative algebraic group $A$
which is a unique extension of an abelian variety by an algebraic torus.  In other words,
there is a defining exact sequence of algebraic groups
\begin{equation}\label{eq:semiabelian-extension}
  0\longrightarrow T\longrightarrow A\longrightarrow A_0\longrightarrow 0,
\end{equation}
where $A_0$ is an abelian variety and $T= \bG_m^r\simeq (\bC^*)^r$ for some $r\geq 0$.
We refer to {\cite[Appendix~A]{Yamanoi2026}} for background on semi-abelian varieties.

An \emph{equivariant compactification} of a semi-algebraic variety $A$ is a complete algebraic variety $\overline A$ together with an open immersion $A\hookrightarrow \overline A$ such that the translation action $A\times A\to A$ extends to an algebraic action
\[
A\times \overline A\longrightarrow \overline A.
\]
Every semi-abelian variety admits a smooth projective equivariant
compactification.  One convenient construction is as follows. Consider the exact sequence \eqref{eq:semiabelian-extension}, choose a smooth projective toric compactification $\overline T$ of $T$, and form the categorical quotient
\[
  \overline A:=(A\times \overline T)/T,
\]
where $t\in T$ acts diagonally on \(A\times \overline T\) by $t\cdot(a,x)=(a-t,t\cdot x)$; see \cite[Lem.~A.2]{Yamanoi2026} for the well definedness of this quotient.  Then $\overline A\to A_0$ is a locally trivial
fiber space with fiber isomorphic to $\overline T$, the open subset $A=A\times^T T$ is dense in $\overline A$, and translations by $A$ extend to $\overline A$.  If $\overline T$ is chosen smooth and projective, then so is
$\overline A$. We refer to \cite{Vojta1999,Yamanoi2026} for more details about such compactifications.

\subsection{Hyperbolicity results}\label{subsec:hyperbolicity-results}

Let $M$ be a complex manifold, let $x\in M$ and let $v\in T_xM$ be a non-zero vector. Denote by $\operatorname{Hol}(\bD,M)$ the set of holomorphic maps from $\bD$ to $M$. Recall that the \emph{Kobayashi--Royden infinitesimal pseudometric} is defined as
\begin{equation}\label{eq:kobayashi-royden}
  F_M(v)
  :=\inf\left\{r>0\ \mid \ \exists f\in \operatorname{Hol}(\bD,M),\ 
  f(0)=x,\ r f'(0)=v\right\}.
\end{equation}
Thi pseudometric is decreased by
holomorphic maps between complex manifolds.  The path pseudodistance obtained by integrating $F_M$ is the Kobayashi pseudodistance, and $M$ is said to be \emph{Kobayashi hyperbolic} if the Kobayashi pseudodistance is a distance. Remark that the existence of an entire curve $f:\bC\to M$ (a non-constant holomorphic map) with $f'(0)=v$ forces $F_M(v)=0$. The above notions also make sense for singular complex spaces. We refer to \cite{Kobayashi1998} for the standard theory of Kobayashi pseudodistance.

Next we introduce some notions that will be used in Yamanoi's theorem. 

\begin{defi}\label{tautness}
Let $W$ be a relatively compact open subset of a complex space $M$, and let $\Delta\subset M$ be closed. We say that $W$ is \emph{tautly embedded modulo $\Delta$ in $M$} if for every
sequence $f_n\in \mathrm{Hol}(\bD,W)$ one of the following alternatives holds:
\begin{enumerate}[label=\textup{(\roman*)}]
\item (\emph{converging subsequence}) there is a subsequence $f_{n_k}$ converging uniformly on compact subsets
of $\bD$ to a holomorphic map $f:\bD\to M$;
\item (\emph{compactly divergent subsequence}) for every compact $K\subset \bD$ and every compact
$L\subset M\setminus\Delta$, there exists $n_0$ such that
$f_n(K)\cap L=\varnothing$ for all $n\geq n_0$.
\end{enumerate}    
\end{defi}

In other words, if $W$ is tautly embedded modulo $\Delta$, then a failure of normality is allowed only by escaping, locally uniformly on the source, toward the exceptional set $\Delta$. 
Let $d_W$ be the Kobayashi pseudodistance of $W$.  We say that $W$ is \emph{hyperbolically embedded modulo $\Delta$ in $M$} if, for any distinct $p,q\in\overline W$ not both in $\Delta$, there are neighborhoods $V_p,V_q\subset M$ such that
\[
 \inf\{d_W(x,y):x\in V_p\cap W,\ y\in V_q\cap W\}>0.
\]
For $\Delta=\varnothing$ this is the usual notion of a hyperbolic embedding.  The theorem of Kiernan--Kobayashi~\cite{KiernanKobayashi1973} states that taut embedding modulo $\Delta$ implies hyperbolic embedding modulo $\Delta$ (see \cite[Th.~5.1.13]{Kobayashi1998}).

A complex algebraic variety $X$ is \emph{Kobayashi hyperbolic modulo $\Delta$} if its Kobayashi pseudodistance satisfies $d_X(x,y)>0$ for distinct $x,y\in X$ not both in $\Delta$.  It is \emph{pseudo-Kobayashi hyperbolic} if this holds for some proper Zariski closed subset $\Delta\subsetneq X$.\\[.1cm]


Let $X\subset A$ be a closed algebraic subvariety of a semi-abelian variety $A$.
Let $\overline A$ be a smooth equivariant compactification, and let
$\overline X$ be the Zariski closure of $X$ in $\overline A$. Following \cite{Yamanoi2026}, we define
\begin{align}\label{eq:yamanoi-special-boundary-set}
  Z_{\overline X}   :=\Bigl\{x\in \overline X \mid {} &\exists\text{ a semi-abelian
 subvariety }B\subset A\text{ such that}\notag\\ 
  &\dim(x+B)\geq 1\ \text{and }x+B\subset \overline X \Bigr\},\\
S_{\overline X}  :=\overline{Z_{\overline X}}^{\mathrm{zar}}  \subset & \overline X \notag
\end{align}
where $\overline{Z_{\overline X}}^{\mathrm{zar}}$ means the Zariski closure of
$Z_{\overline X}$ in $\overline X$.

The following theorem is the main engine of our paper: 
\begin{thm}[{Yamanoi~\cite[Th.~1.3, Cor.~1.4]{Yamanoi2026}}]\label{thm:pkob}
Let $X\subset A$ be a closed algebraic subvariety.
Let $\overline A$ be a smooth equivariant compactification and let
$\overline X\subset \overline A$ be the compactification.
Then $X$ is tautly embedded modulo $S_{\overline X}$ in $\overline X$. 

Moreover, if $X$ is of log-general type, then $S_{\overline X}$ is a proper Zariski closed subset of $\overline X$. Consequently, $X$ is pseudo-Kobayashi hyperbolic
\end{thm}

Now we explain some structure results of semi-abelian varieties which will be important when we apply Yamanoi's theorem. If $X\subset A$ is an irreducible closed subvariety, define its \emph{stablizer} by
\[
  \operatorname{Stab}_A(X):=\left\{a\in A \mid{} a+X=X\right\}.
\]
Then the connected component of identity $\operatorname{Stab}_A^0(X)$ is a semi-abelian subvariety. Denote $B:=\operatorname{Stab}_A^0(X)$ and let $q:A\to A/B$ be the quotient map. Let $Y=q(X)$. 
Recall the definition of the logarithmic Kodaira dimension: Take a resolution $X'\to X$ and a smooth projective compactification $\overline X'$ with reduced simple normal crossings boundary $E=\overline X'\setminus X'$.  Set $\overline\kappa(X):=\kappa(\overline X',K_{\overline X'}+E)$,
where the right-hand side is the Iitaka dimension (see
\cite[Def.~2.1.3]{Lazarsfeld2004}). Note that $\overline\kappa(X)$ is independent of these choices.  The variety $X$ is of \emph{log-general type} if $\overline\kappa(X)=\dim X$. 

The following structure theorem is due to Ueno in the abelian case and Noguchi in the semi-abelian case. We will use Abramovich's formulation~\cite[Introduction, Th.~3]{Abramovich1994}.

\begin{thm}[Kawamata--Abramovich, \cite{Kawamata1980,Abramovich1994}]\label{thm:kawamata-abramovich}
The fibers of $q\vert_X:X\to Y$ are translates of $B$, the stabilizer of $Y$ in $A/B$ is finite, and
\[
\overline\kappa(X)=\dim Y=\dim X-\dim B.
\]
In particular, $Y$ is of log-general type, and $X$ is of log-general type if and only if $B=0$.
\end{thm}

We define the special locus of $X$ with respect to $A$ as
\begin{equation}\label{eq:definition-special-locus}
\operatorname{Sp}_A(X)
  :=\left\{x\in X \mid{} \exists B'\subset A\text{ semi-abelian},\
      \dim B'>0,\ x+B'\subset X\right\}.
\end{equation}
Theorem~\ref{thm:kawamata-abramovich} implies that $\operatorname{Sp}_A(X)$ is Zariski closed in $X$. Each irreducible component $Z$ of $\operatorname{Sp}_A(X)$ is stabilized by a positive-dimensional semi-abelian
subvariety $B_Z=\operatorname{Stab}_A^0(Z)$, and $Z/B_Z$ is of log-general type. By \eqref{eq:yamanoi-special-boundary-set}, $Z_{\overline X}$ consists of points on positive-dimensional semi-abelian subvarieties contained in $\overline X$, and $S_{\overline X}$ is its Zariski closure, the exceptional set in Theorem~\ref{thm:pkob}.  Notice that $Z_{\overline X}\cap X=\operatorname{Sp}_A(X)$ and $S_{\overline X}\cap X=\operatorname{Sp}_A(X)$.
The second equality follows because $X$ is open in $\overline X$ and $\operatorname{Sp}_A(X)$ is closed in $X$.  The additional points of $S_{\overline X}$ lie in $\overline X\setminus X$.\\[.1cm]


%

The study of hyperbolicity of subvarieties of (semi-)abelian varieties goes back to the classical Bloch--Ochiai theorem~\cite{Bloch1926,Ochiai1977}: The Zariski closure of an entire curve in an abelian variety is a translate of an abelian subvariety.  Consequently, every proper closed subvariety of a \emph{simple} abelian variety admits no nonconstant entire curve, and thus is Brody hyperbolic.  By Brody's theorem it is also Kobayashi hyperbolic~\cite{Brody1978,Kobayashi1998}.

In the abelian case, the Yamanoi's exceptional set is nothing but the usual \emph{special set} of $X$, which is defined as
\[
  \mathrm{Sp}(X)=\left\{x\in X \mid{} \exists\text{ an abelian subvariety }
  B\subset A,\ \dim B>0,\ x+B\subset X\right\}.
\]
Previously, Yamanoi proved in \cite{Yamanoi2019} that every closed subvariety
$X\subset A$ is taut modulo $\mathrm{Sp}(X)$ and hence Kobayashi hyperbolic
modulo $\mathrm{Sp}(X)$.  Kawamata's theorem implies that if $X$ is of general
type, then $\mathrm{Sp}(X)$ is a proper Zariski closed subset. Therefore $X$ is
pseudo-Kobayashi hyperbolic.  Conversely, Ueno's structure theorem gives
$\mathrm{Sp}(X)=X$ when $X$ is not of general type \cite{Yamanoi2019}.
Thus, for closed subvarieties of abelian varieties, being of general type is equivalent to
pseudo-Kobayashi hyperbolicity.

The above picture was generalized to the semi-abelian case by Yamanoi in \cite{Yamanoi2026}. In fact, from Theorem~\ref{thm:pkob} one can extract a local lower bound of the Kobayashi–Royden pseudometric, which will be a key ingredient in the proof of our main theorem.

\begin{cor}\label{cor_yamanoi-metric}
Let $Z$ be a positive-dimensional closed subvariety of a semi-abelian variety, and assume that $Z$ is of logarithmic general type.  Then there is a proper Zariski closed subset $S_Z\subsetneq Z$ such that the following holds: For every compact set $K\subset Z_{\mathrm{reg}}\setminus S_Z$ and every Hermitian norm $\|\cdot\|$ on a neighborhood of $K$, there exists a constant $c_K>0$ such that
\begin{equation}\label{F_Z-ineq}
 F_Z(z;v)\ge c_K\|v\|   
\end{equation}
for all $z\in K$ and $v\in T_zZ$.
\end{cor}
\begin{proof}
Let $S_{\overline Z}$ be the same as in Theorem~\ref{thm:pkob} and let $S_Z := S_{\overline Z} \cap Z$. From Theorem~\ref{thm:pkob} we know that $Z$ is tautly embedded modulo $S_{\overline Z}$ in $\overline Z$. Fix a compact subset \(K\subset Z_{\mathrm{reg}}\setminus S_Z\). 
 
We now argue by contradiction. Suppose that \eqref{F_Z-ineq} fails. Then, by the definition of Kobayashi–Royden infinitesimal pseudometric, derivatives at the origin are \emph{not} uniformly bounded among holomorphic disks $\varphi:\mathbb D\to Z$ with $\varphi(0)\in K$.

Choose a sequence \(\{\varphi_j\}\) with \(\|\varphi_j'(0)\|\to +\infty\).  This sequence cannot satisfy the second alternative in Definition~\ref{tautness}, as every \(\varphi_j\) meets the fixed compact set \(K\) at the origin.  Hence, after passage to a subsequence, \(\{\varphi_j\}\) converges uniformly on compact subsets in $\overline{Z}$.  Passing again to a subsequence, $\{\varphi_j(0)\}$ converges to some point $z_\infty\in K$.
The limiting map therefore takes \(0\) to the interior smooth point \(z_\infty\).  Choose a neighborhood \(W\Subset Z_{\mathrm{reg}}\setminus S_Z\) of \(z_\infty\).  Uniform convergence implies that, for some \(0<r<1\) and all sufficiently large \(j\), one has \(\varphi_j(\mathbb D_r)\subset W\). Then $\|\varphi_j'(0)\|$ will converge to the derivative norm of the limiting map at the origin, which cannot be $+\infty$. A contradiction.
\end{proof}

\subsection{Unit maps and quasi-Albanese morphisms}\label{unit_map_sect}

Now we come back to the setting of Theorem~\ref{mainthm:rationally-connected} and \ref{mainthm:rationally-connected_Kob}: Let $X$ be a rationally connected normal complex projective variety and \(f\) is a polarized divisorial PCF endomorphism on it. We continue to use the notations introduced in Sect.~\ref{sec:PCF-rc-chow}.

Recall that we have a tower of finite etale morphisms 
\[
\cdots \xrightarrow{f} U_{n+1} \xrightarrow{f} U_n \xrightarrow{f} \cdots \to U_0=U,    
\]
where $U_n$ is the complement of $f^n$-preimage of $\Delta(f):= \Sing(X) \cup P^{(1)}_f$. Consider \(K_{\infty} \subset K:=\bC(X)\) the subfield generated over
\(\bC\) by the elements of \(\mathcal O(U_n(f))^*\) for all $n \ge 0$. 

Let $r_n$ be the number of prime components of the reduced part of $f^{-n}(\Delta(f))$.  In this subsection we assume that $\{r_n\}$ is unbounded and that $\operatorname{trdeg}_{\bC}K_\infty=N:=\dim X$.  Under these assumptions we construct the unit maps needed in Section~\ref{sec:main-proof-hyperbolicity}.



Now we choose $n$ such that
\[
\operatorname{trdeg}_{\bC} K_n=N,
 \qquad \rho:=\operatorname{rank} M_n>N,
\]
and choose a basis $u_1,\dots,u_\rho$ of $M_n$.  It defines the \emph{unit map}
\[
 \alpha:U_n\longrightarrow T:=(\bC^*)^\rho,
 \qquad \alpha=(u_1,\dots,u_\rho).
\]
Let $Y_1=\overline{\alpha(U_n)}^{\Zar}$.  Then $\dim Y_1=N$ and $\alpha$ is
generically finite.  

\begin{lem}\label{lem:unit-nondegenerate}
The variety $Y_1$ is nondegenerate, i.e. it is not contained in a translate of a proper subtorus of $T$.
\end{lem}

\begin{proof}
If $Y_1\subset aT'$ for a proper subtorus $T'$, a nontrivial character of $T$
which is trivial on $T'$ is constant on $Y_1$.  Composing with $\alpha$ gives a
nontrivial integral multiplicative relation among the basis units modulo
constants. A contradiction.
\end{proof}

Let $A=\Stab(Y_1)^0$, let $q_T:T\to T/A$ be the quotient, and define
\[
 Z=q_T(Y_1),
 \qquad \beta=q_T\circ\alpha:U_n\to Z.
\]
The stabilizer of $Z$ is finite.  We claim that $\dim Z>0$. If $Z$ were a point,
$Y_1$ would be a translate of $A$, a proper subtorus because
$\rho>N=\dim Y_1$, contradicting Lemma~\ref{lem:unit-nondegenerate}.  By Theorem~\ref{thm:kawamata-abramovich}, $Z$ is of logarithmic general type.

The map $\beta$ may have positive-dimensional general fibers.  Put $E:=\beta^*\bC(Z)$, $E_j:=(f^j)^*E$, $F_m:=\bC(E_0,\dots,E_m)$, and $F_\infty:=\bigcup_{m\ge0}F_m$.  For the construction of a generically finite map into a torus, we assume in addition that $\operatorname{trdeg}_{\bC}F_\infty=N$.


Choose $0\le j_1<\cdots<j_s$ such that the fields $E_{j_1},\dots,E_{j_s}$
generate a field of transcendence degree $N$, and choose $M\ge n+j_s$.  Then
\begin{equation}\label{Psi_unit_map}
  \Psi=\bigl(\beta\circ f^{j_1},\dots,\beta\circ f^{j_s}\bigr):
 U_M\longrightarrow Z^s   
\end{equation}
is \emph{generically finite} onto its image.  Hence its differential is injective away
from a proper algebraic subset $R\subsetneq U_M$.  Define
\[
 A_0:=R\cup\bigcup_{\nu=1}^s
 (\beta\circ f^{j_\nu})^{-1}(S_Z\cup\Sing Z),
\]
where $S_Z$ is the proper algebraic subset appeared in Coroallary~\ref{cor_yamanoi-metric}.
Thus $A_0$ is proper.

\begin{prop}\label{prop:local-coercivity}
For every compact $K\Subset U_M\setminus A_0$ there is $c_K>0$ such that
\[
 F_{U_M}(x;v)\ge c_K\|v\|
\]
for all $x\in K$ and $v\in T_xU_M$.
\end{prop}

\begin{proof}
Apply Corollary~\ref{cor_yamanoi-metric} to the compact image of $K$ under each coordinate
of $\Psi$.  The Kobayashi--Royden metric on a product is the maximum of the
factor metrics.  The distance-decreasing property under holomorphic maps gives
\[
 F_{U_M}(x;v)\ge F_{Z^s}(\Psi(x);d\Psi_xv).
\]
Since $d\Psi$ is injective on $K$, compactness gives $c'>0$ with
$\|d\Psi_xv\|\ge c'\|v\|$.  Combining the estimates proves the statement.
\end{proof}

\subsubsection{Quasi-Albanese morphisms}
\label{subsec:quasi-albanese}

The unit map $ \alpha:U_n\longrightarrow T:=(\bC^*)^\rho$ has an intrinsic interpretation.  Choose a
log-smooth compactification $U_n=\widehat X_n\setminus \widehat B_n$, where $\widehat X_n$ is smooth projective and $\widehat B_n$ is simple
normal crossing.  Since $\widehat X_n$ is rationally
connected, one has
$\operatorname{Alb}(\widehat X_n)=0$ and
$H^0(\widehat X_n,\Omega^1_{\widehat X_n})=0$.  Hence the quasi-Albanese
variety $\operatorname{Alb}_{\rm q}(U_n)$ is a torus of dimension
\[
 \overline q(U_n):=
 h^0\!\left(\widehat X_n,
 \Omega^1_{\widehat X_n}(\log\widehat B_n)\right);
\]
see \cite[Sections~2--3]{FujinoQuasiAlb} and
\cite{FujinoQuasiAlbCorrection}.

Fix $x_0\in U_n$ and normalize the quasi-Albanese morphism
$a_n:U_n\to \operatorname{Alb}_{\rm q}(U_n)$ by $a_n(x_0)=1$.  Its universal property, applied to the
normalized units $u/u(x_0):U_n\to \bC^*$, gives a canonical identification between the character lattice of $\operatorname{Alb}_{\rm q}(U_n)$ and $\mathcal{O}(U_n)^*/\bC^*=M_n$.

Thus, if $[u_1],\ldots,[u_\rho]$ is the basis of $M_n$ used above, the
corresponding characters identify $\operatorname{Alb}_{\rm q}(U_n)$ with $(\bC^*)^\rho$. Up to translation and such a torus isomorphism, \emph{the unit map $\alpha$ is the
quasi-Albanese morphism}.  In terms of differential forms, we have
\[
 a_n^*\!\left(\frac{dz_i}{z_i}\right)=d\log u_i,
 \qquad
 H^0\!\left(\widehat X_n,
 \Omega^1_{\widehat X_n}(\log\widehat B_n)\right)
 =\bigoplus_{i=1}^{\rho}\bC\,d\log u_i.
\]

\section{Proof of the main results}\label{sec:main-proof-hyperbolicity}

In this section, we give the proofs of Theorem~\ref{mainthm:rationally-connected} and \ref{mainthm:rationally-connected_Kob}.

\subsection{Backward-orbit density and Kobayashi hyperbolicity}
We start with some lemma:
\begin{lem}
\label{lem:finite-closure}
Let $h:V\to W$ be a finite dominant morphism of irreducible varieties, with
$W$ normal.  Then for every subset $S\subset W$,
\[
 h^{-1}(\overline S^{\Zar})=\overline{h^{-1}(S)}^{\Zar}.
\]
\end{lem}



\begin{proof}
Since $W$ is normal, the going-down theorem implies that $h$ is open.
Indeed, the image of an open subset of $V$ is constructible by Chevalley's theorem and stable under generalization by going down. Hence it is open by \cite[IV$_1$, Cor.~1.10.4]{EGAIV}.

The inclusion \(\overline{h^{-1}(S)}^{\Zar}\subset h^{-1}(\overline S^{\Zar})\) is clear. Conversely, let $v\in h^{-1}(\overline S^{\Zar})$ and let $O$ be an open neighborhood of $v$. Since $h$ is open, $h(O)$ is an open neighborhood of $h(v)$ and therefore intersects $S$. Hence $O\cap h^{-1}(S)\ne\varnothing$, which proves the reverse inclusion.
\end{proof}


Recall that $X$ is a normal rationally connected projective variety of dimension $N$, and $f:X\to X$ is a $q$-polarized divisorially PCF endomorphism on $X$, with $f^*H\sim qH$ for an ample Cartier divisor $H$ and an integer $q\ge2$.  As in \eqref{eq:definition-adjoin-locus}, we write $\Delta_f=\Sing(X)\cup P_f^{(1)}$ and $U_r=X\setminus f^{-r}(\Delta_f)$. Thus $f^k:U_{r+k}\to U_r$ is finite \'etale.

\begin{lem}\label{lem:backward-density}
For every $p\in U_0$, the set
\[
 \mathcal O^-(p):=\bigcup_{k\ge0}f^{-k}(p)
\]
is Zariski dense in $X$. 
Moreover, for each $k_0 \ge 0$,
the truncation $\bigcup_{k\ge k_0}f^{-k}(p)$ is also Zariski dense.
\end{lem}

\begin{proof}
Let $W=\overline{\mathcal O^-(p)}^{\Zar}$. Since $X$ is normal, by
Lemma~\ref{lem:finite-closure} we obtain that $f^{-1}(W)= \overline{f^{-1}(\mathcal O^-(p))}^{\Zar} = \overline{\bigcup_{k\ge 1}f^{-k}(p)}^{\Zar} \subset W$. The descending chain
$W\supset f^{-1}(W)\supset\cdots$ stabilizes and we obtain a nonempty closed set $W_\infty$ satisfying
$f^{-1}(W_\infty)=W_\infty$. A tail of the backward orbit is dense in
$W_\infty$.

Suppose for contradiction that $W_\infty\ne X$. As a finite morphism preserves dimension, every maximal-dimensional irreducible component of $W_\infty$ is mapped by $f$ onto another one. Since $f\vert_{W_\infty}$ is surjective, the induced map on the finite set of maximal-dimensional components is surjective and hence
a permutation. Let $V$ be one such component. We will replace $f$ by an iterate $F=f^a$ such that $F(V)=V$ and no other maximal-dimensional component maps onto $V$. Let $B\subsetneq V$ be the union of the images of all irreducible components of $F^{-1}(V)$ that do not dominate $V$.  

Suppose for contradiction that $V\not\subset\Delta_f$. Let $y\in V\setminus(\Delta_f\cup B)$ be a general point. Then every point of $F^{-1}(y)$ lies on the unique component of $F^{-1}(V)$ dominating $V$, namely $V$ itself. Note that these points lie in the smooth locus and outside $P_f$ while the divisorial branch locus of $F$ lie in $P_f$. Applying the Zariski--Nagata purity theorem \cite[Th.~X.3.1]{SGA1}, we deduce that $F$ is unramified over $y$. Thus the fiber has $\deg F=q^{aN}$ distinct points, which are all contained in $V$. On the other hand, as $F^*H\sim q^a H$, we have, by the projection formula, that $\deg(F\vert_V)=q^{ar}$ where $r:=\dim V<N$. A contradiction. Hence every maximal-dimensional component of $W_\infty$ lies in $\Delta_f$.

Since $\Delta_f$ is forward invariant and $p\notin\Delta_f$, we know that the backward orbit of $p$ is disjoint from $\Delta_f$. But by above arguments, a tail must intersect some nonempty open subset of each maximal component of $W_\infty$ which is contained in $\Delta_f$. This contradiction proves $W=X$. The above proof works for any $k_0$-truncation of the backward orbit. 
\end{proof}

Next we use the Zariski density of backward orbits of points in $U$ to prove a "dense-or-nothing" principle for the degeneracy locus of Kobayashi--Royden metric on $U$.

We choose a smooth Hermitian metric on the regular locus of $X$. Denote by $\|\cdot\|$ the induced norm on the tangent bundle, and use its restriction to each $U_r$.  Note that local equivalence of Hermitian norms shows that the degeneracy locus below does not depend on this choice.

\begin{prop}\label{prop:dynamical-amplification}
Let $r\geq 0$. Consider the degeneracy locus
\[
 \Sigma_r:=\left\{x\in U_r \mid \begin{array}{l}
 \text{for every neighborhood }V\ni x,\\[-1mm]
 \inf\{F_{U_r}(y;v)\vert y\in V,\ \lVert v\rVert=1\}=0
 \end{array}\right\}.
\]
The set $\Sigma_r$ is either empty or Zariski dense in $U_r$.
\end{prop}

\begin{proof}
We claim that $\Sigma_r$ is preserved under $f^{-k}$ for $k \geq 0$. Indeed, first, if $x\in\Sigma_r$ and $y\in f^{-k}(x)$, then
$y\in U_{r+k}\subset U_r$. Second, consider the finite étale covering $f^k:U_{r+k}\to U_r$. By using the invariance of the Kobayashi--Royden metric under unramified coverings and monotonicity under the inclusion $U_{r+k}\hookrightarrow U_r$, we get that $y \in \Sigma_r$. Thus $f^{-k}(\Sigma_r)\subset\Sigma_r$ for $k \geq 0$.
 
We remark that if there exists a proper algebraic subset $A\subsetneq U_r$ such that $\Sigma_r \subset A$, then $\Sigma_r=\emptyset$ and $U_r$ (hence $U_0$) is Kobayashi hyperbolic. The reason is as follows. The full backward orbit of any $x\in\Sigma$ cannot be contained in some proper algebraic set $A$, because otherwise it would contradict Lemma~\ref{lem:backward-density}. Hence
$\Sigma$ is either empty or Zariski-dense. 
\end{proof}

Now we prove Theorem~\ref{mainthm:rationally-connected_Kob}.

\begin{proof}[Proof of Theorem~\ref{mainthm:rationally-connected_Kob}]

Recall that \(r_n\) is the number of prime components of the reduced part of $f^{-n}(\Delta(f))$. If $\{r_n\}$ is bounded, then Corollary~\ref{cor:divisorial-tower} produces an invariant divisor $D_{\infty}$, which gives us the Case~\textup{(1)}.

Suppose $\{r_n\}$ is unbounded.  If $\operatorname{trdeg}_{\bC}K_\infty<N:=\dim X$, Proposition~\ref{prop:pcf-rational-fibration} gives Case~\textup{(2)}.  Otherwise we construct $Z$, $E_j$ and $F_m$ as in Section~\ref{unit_map_sect}.

If $t:=\operatorname{trdeg}_{\bC}F_\infty<N$, we choose $m_0$ with $\operatorname{trdeg}_{\bC}F_{m_0}=t$.  Since $\dim Z>0$, one has $1\le t<N$.  Let $L$ be the relative algebraic closure of $F_{m_0}$ in $\bC(X)$.  It is a finite extension of $F_{m_0}$, and $f^*F_{m_0}\subset F_{m_0+1}$ with $F_{m_0+1}/F_{m_0}$ algebraic implies $f^*L\subset L$, exactly as in Lemma~\ref{lem:invariant-field}. Now we can use Lemmas~\ref{lem:invariant-subfield-chow} and \ref{lem:polarization-descent} and Proposition~\ref{prop:pcf-descent} to give the fibration in Case~\textup{(2)}. Its base is rationally connected as a dominant rational image of $X$.

It remains to consider $\operatorname{trdeg}_{\bC}F_\infty=N$.  The construction \eqref{Psi_unit_map} and Proposition~\ref{prop:local-coercivity} put $\Sigma_M$ inside the proper algebraic subset $A_0\subset U_M$.  Proposition~\ref{prop:dynamical-amplification} then gives $\Sigma_M=\varnothing$ and the Kobayashi hyperbolicity of $U_0$.

\end{proof}

\subsection{Bubbling and tautness}
We now assume that $X$ is smooth, rationally connected, and $f$ is a
$q$-polarized PCF endomorphism. 

In this section, we prove the following criterion for tautness, which will finish the proof of Theorem~\ref{mainthm:rationally-connected}.

\begin{thm}[Criterion for tautness]\label{tautness-criterion}
Suppose that for some $M\ge 0$ there is a morphism
\[
 \Theta:U_M\longrightarrow W
\]
which is generically finite onto a closed subvariety $W$ of an algebraic torus
and whose stabilizer is finite.  Then $U=U_0$ is taut.
\end{thm}

\begin{proof}[Proof of Theorem~\ref{mainthm:rationally-connected}]
If $\{r_n\}$ is bounded, Corollary~\ref{cor:divisorial-tower} gives a reduced divisor $D_\infty$ containing the ramification locus with $f^{-1}(D_\infty)=D_\infty$.  Thus $f:X\setminus D_\infty\to X\setminus D_\infty$ is \'etale, and Meng--Zhong's theorem~\cite[Th.~1.5]{MengZhong} implies that $(X,D_\infty)$ a toric pair.
Note that every unit on the torus $T$ is a constant times a character. Thus $f|_T$ is a translation composed with a torus homomorphism. By the finiteness and surjectivity, we know that this homomorphism an isogeny, which gives Case~\textup{(1)}.

Suppose $\{r_n\}$ is unbounded.  If either $\operatorname{trdeg}_{\bC}K_\infty<N$ or, after constructing $F_\infty$ as in Section~\ref{unit_map_sect}, $\operatorname{trdeg}_{\bC}F_\infty<N$, the corresponding field constructed in the proof of Theorem~\ref{mainthm:rationally-connected_Kob} gives Case~\textup{(2)}.  Otherwise, we can get a generically finite morphism $\Theta:U_M\to W$ from \eqref{Psi_unit_map}, where $W$ is the Zariski closure of $\Psi(U_M)$ in the product torus.  Each coordinate projection $W\to Z$ is dominant.  Consequently the stabilizer of $W$ is contained in the product of the finite stabilizers of $Z$, and thus is finite. Note that $W$ is of log-general type by Theorem~\ref{thm:kawamata-abramovich}. Now apply Theorem~\ref{tautness-criterion} to obtain the tautness of $U$.
\end{proof}

The rest of this section is devoted to the proof of Theorem~\ref{tautness-criterion}.

We first take a smooth projective equivariant
compactification of the torus and let $\overline W$ be the closure of $W$ there. By Theorem~\ref{thm:pkob}, there exists a proper algebraic subset
$S\subsetneq\overline W$ such that for every sequence of holomorphic disks $\left\{ h_\nu:\bD \to W \right\}$ with $h_\nu(0)\to w\in W\setminus S$, a subsequence converges uniformly on compact subsets as maps to $\overline W$.

\begin{lem}[Uniform energy bound]\label{uniform_energy_bound}
Let $V\subset X$ be Zariski open, let $\theta:V\to W$ be generically finite, and
let $H$ be ample on $X$.  There is a proper algebraic subset
$\Sigma_\theta\subsetneq V$ such that, if there is a sequence of holomorphic disks $\left\{ g_\nu:\bD\to V \right\}$ with $g_\nu(0)\to y\in V\setminus\Sigma_\theta$ and $\theta\circ g_\nu$ converges in a projective compactification of $W$, then
for every $0<r<1$,
\begin{equation}\label{eq:area-bound-general}
 \sup_\nu\int_{\bD_r}g_\nu^*\omega_H<\infty,
\end{equation}
where $\omega_H$ is a positive curvature form representing $c_1(H)$.    
\end{lem}

\begin{proof}
We first resolve the indeterminacy locus of the extended rational map of $\theta$ between $X$ and $\overline W$ and get $ \widetilde X\xrightarrow{\mu}X$, $\widetilde X\xrightarrow{\widetilde\theta}\overline W$ with $\mu$ an isomorphism over $V$.  If $L$ is ample on $\overline W$, then
$\widetilde\theta^*L$ is big because $\theta$ is generically finite.  By Kodaira's lemma, there is an effective divisor $E$ such that
\begin{equation}\label{L-H-E}
 a\widetilde\theta^*L\sim\mu^*H+E.
\end{equation}
Let $\Sigma_\theta$ be the image of
$E\cap\mu^{-1}(V)$ in $V$.  

To prove the estimate, we use the Nevanlinna's order function $T_{g_\nu,H}(r):= \int^r_0 \frac{dt}{t} \int_{\bD_t}g_\nu^*\omega_H$, $0<r < 1$. Then from \eqref{L-H-E} and the Poincar\'e--Lelong formula we obtain $\mu^*\omega_H
=
a\widetilde\theta^*\omega_L-[E]+dd^c\log\|s_E\|$, where $s_E$ is the  canonical section of $\mathcal{O}(E)$. Pulling this identity back by $\widetilde g_\nu:\bD \to\widetilde X$ and integrating it, Jensen's formula yields
\[
 T_{g_\nu,H}(R)
=
aT_{\theta\circ g_\nu,L}(R)
-
N_{\widetilde g_\nu,E}(R)
+
\frac{1}{2\pi}\int_0^{2\pi}
\log\|s_E(\widetilde g_\nu(Re^{i\vartheta}))\|\,d\vartheta
-
\log\|s_E(\widetilde g_\nu(0))\|.   
\]
Note that the counting function $N_{\widetilde g_\nu,E}(R)\ge0$ because $E$ is effective. Since $\widetilde X$ is compact, $\log\|s_E\|$ is uniformly bounded above, while the assumption $g_\nu(0)\longrightarrow y\notin \mu(E\cap\mu^{-1}(V))$ implies that $\widetilde g_\nu(0)$ stays in a compact subset of $\widetilde X\setminus E$. Thus $\log\|s_E(\widetilde g_\nu(0))\|$ is uniformly bounded below.  Consequently, for $R<1$
\begin{equation}\label{FMT}
 T_{g_\nu,H}(R)
 \le aT_{\theta\circ g_\nu,L}(R)+C_R,
\end{equation}
where $C_R$ is independent of $\nu$.  
By the convergence of $\left\{ \theta \circ g_{\nu}\right\}$, we know that the right hand side of \eqref{FMT} is bounded above by a uniform constant.

Finally, for $0<r<R<1$, one has
\[
 T_{g_\nu,H}(R) \ge \int^R_r \frac{dt}{t} \int_{\bD_t}g_\nu^*\omega_H
 \ge \int^R_r \frac{dt}{t} \int_{\bD_r}g_\nu^*\omega_H
 =\log(R/r)\int_{\bD_r}g_\nu^*\omega_H.
\]
Together with the upper bound of $T_{g_\nu,H}(R)$, we have proved \eqref{eq:area-bound-general}.
\end{proof}

We will apply Lemma~\ref{uniform_energy_bound} to $V=U_M$ and $\theta=\Theta$, and put
\[
 \Sigma:=\Sigma_\Theta\cup\Theta^{-1}(S).
\]
Its closure in $X$ is proper.

\begin{proof}[Proof of Theorem~\ref{tautness-criterion}]

Assume for contradiction that $U$ is not taut.  Choose
$\varphi_\nu:\bD\to U$ with no subsequence converging in
$\operatorname{Hol}(\bD,U)$ and no compactly divergent subsequence (cf. Definition~\ref{tautness}).  Since the sequence is
not compactly divergent, after composing with automorphisms of the disc and
passing to a subsequence we may arrange $\varphi_\nu(0)\longrightarrow p\in U$. 

Now we make use of the smoothness of $X$. Recall the generalized form of Hurwitz theorem (see {\cite[Theorem~3.6.11]{Kobayashi1998}}):

\begin{thm}\label{Hurwitz}
Let $Z$ be a complex analytic space and $S$ a \emph{Cartier divisor} in $Z$. Assume that a sequence $\left\{g_m\right\} \subset \operatorname{Hol}(\bD, Z \setminus S)$ converges to a map $g \in \operatorname{Hol}(\bD, Z)$. Then $g(\bD)$ is either in $Z \setminus S$ or in $S$.
\end{thm}

In our setting, we need the smoothness of $X$ to conclude that the PCF divisor $P_f$ is Cartier. In the following arguments, we repeatedly use the generalized Hurwitz theorem: if
$g_\nu:\Omega\to X\setminus P_f$ converges locally uniformly to
$g:\Omega\to X$ on a connected domain, then either $g(\Omega)\subset P_f$ or
$g(\Omega)\cap P_f=\varnothing$.

By the density in Lemma~\ref{lem:backward-density}, one can choose $k\ge M$ and $y\in f^{-k}(p)\setminus\Sigma$. For large $\nu$, a local inverse branch of $f^k:U_k\to U$ gives points
$y_\nu\in f^{-k}(\varphi_\nu(0))$ with $y_\nu\to y$.  Since $\bD$ is simply
connected, $\varphi_\nu$ lifts uniquely to $\psi_\nu:\bD\to U_k\subset U_M$ satisfying $f^k\circ\psi_\nu=\varphi_\nu$ and $\psi_\nu(0)=y_\nu$.
\[
\begin{tikzcd}
  \mathbb{D} \ar[r, "\psi_{\nu}"] \ar[rd, "\varphi_{\nu}"] & U_k \ar[d, "f^k"] \ar[r, hook] & U_M \\
  & U &
\end{tikzcd}
\]

Note that the sequence $\psi_\nu$ has \emph{no} subsequence converging locally uniformly as maps
to $X$:  If $\psi_{\nu_j}\to\psi$, then $\psi(0)=y\notin P_f$. Thus Theorem~\ref{Hurwitz} gives
$\psi(\bD)\subset U$.  Consequently
$\varphi_{\nu_j}\to f^k\circ\psi$ in $X$. Using the Hurwitz theorem again, plus the fact that $p\in U$, we obtain that $(f^k\circ\psi)(\bD)\subset U$, contrary to
the choice of $\varphi_\nu$. By passing to a subsequence and using the diagonal argument, one can find some $0 < r <1$ such that $\psi_\nu$ has no convergent subsequence on $\bD_r$ as maps to $X$.

Since $y\notin\Theta^{-1}(S)$, Theorem~\ref{thm:pkob} gives a subsequence for which
$\Theta\circ\psi_\nu$ converges in $\overline W$.  Since
$y\notin\Sigma_\Theta$, by Lemma~\ref{uniform_energy_bound} one can choose $r < R < 1$ and find a
constant $C$ such that
\begin{equation}\label{eq:psi-area-C}
 \int_{\bD_R}\psi_\nu^*\omega_H\le C
\end{equation}
for all $\nu$, while $\psi_\nu$ has no convergent subsequence on $\bD_r$.

Choose $\ell\ge M$ so large that $q^\ell>C$. Using the density in Lemma~\ref{lem:backward-density}, after increasing $\ell$ if necessary, one can choose $z\in f^{-\ell}(y)\setminus\Sigma$. Now we choose $z_\nu\in f^{-\ell}(y_\nu)$ with $z_\nu\to z$, and lift the maps
$\psi_\nu:\bD\to U$ through the covering $f^\ell:U_\ell\to U$ to $\chi_\nu:\bD\to U_\ell\subset U_M$ satisfying $f^\ell\circ\chi_\nu=\psi_\nu$ and $\chi_\nu(0)=z_\nu$.
\[
\begin{tikzcd}
  & & & U_{\ell} \ar[d, "f^{\ell}"] \ar[ld, hook] \\
  \bD \ar[urrr, "\chi_\nu", bend left=15] \ar[r, "\psi_{\nu}"] \ar[rd, "\varphi_{\nu}"] & U_k \ar[d, "f^k"] \ar[r, hook] & U_M \ar[d, "\Theta"] \ar[r, hook] & U \\
  & U & W &
\end{tikzcd}
\]

The same argument as before shows that $\chi_\nu$ has no convergent subsequence on
$\bD_r$ as maps to $X$.  Since $z\notin\Sigma$, Theorem~\ref{thm:pkob} and Lemma~\ref{uniform_energy_bound} give a uniform $H$-area bound for $\chi_\nu$ on every smaller disc.

Next we will use the non-normality of $\left\{ \chi_\nu \right\}$ and Brody's reparametrization argument to construct a "bubble".

Choose $r<r'<R <1$.  Fix Hermitian metrics on $X$ and on the disc. Consider the weighted derivative norm $M_\nu:=\max_{|z|\le r'}(r'-|z|)\,\|d\chi_\nu(z)\|$. We claim that after passing to a subsequence, $M_\nu \longrightarrow\infty$. Otherwise, $\|d\chi_\nu\|$ will be uniformly bounded on $\bD_r$. Then by Arzel\`a--Ascoli theorem, there exists a subsequence converging on $\bD_r$, contradicting the non-normality of $\left\{ \chi_\nu \right\}$ on $\bD_r$.

Now we choose $a_\nu\in\bD_{r'}$ where this maximum is attained, put
$\lambda_\nu:=\|d\chi_\nu(a_\nu)\|$, and set $g_\nu(\zeta):=\chi_\nu\!\left(a_\nu+\frac{\zeta}{\lambda_\nu}\right)$. Then $(r'-|a_\nu|)\lambda_\nu=M_\nu$, $g_\nu$ is defined on $\bD_{M_\nu}$ and
$\|dg_\nu(0)\|=1$.
Moreover, for $|\zeta|\le M_\nu/2$, the maximality of
$a_\nu$ gives
\[
 \|dg_\nu(\zeta)\|
 =\frac{1}{\lambda_\nu}
   \left\|d\chi_\nu\!\left(a_\nu+\frac{\zeta}{\lambda_\nu}\right)\right\|
 \le 2.
\]
Hence the Brody reparametrization argument and Arzel\`a--Ascoli theorem yield,
after passing to a subsequence, $\left\{g_\nu\right\}$ converges to an entire curve $g:\bC\longrightarrow X$ with $\|dg(0)\|=1$.

Note that Lemma~\ref{uniform_energy_bound} gives a
constant $C'>0$ such that
\[
 \int_{\bD_R}\chi_\nu^*\omega_H\le C'
\]
for all $\nu$.  Therefore, for every fixed $A>0$ and all sufficiently
large $\nu$,
\[
 a_\nu+\lambda_\nu^{-1}\Delta_A\subset\Delta_{r'}\subset\Delta_R,
\]
and hence, by lower semicontinuity of area,
\[
 \int_{\Delta_A}g^*\omega_H
 \le
 \liminf_{\nu\to\infty}
 \int_{a_\nu+\lambda_\nu^{-1}\Delta_A}\chi_\nu^*\omega_H
 \le C'.
\]
Letting $A\to\infty$ shows that $g$ has finite total energy, i.e. $\int_{\bC}g^*\omega_H<\infty$. By the removable singularity theorem, $g$ extends holomorphically
across $\infty$, giving a rational curve in $X$:
\[
 g:\bP^1\longrightarrow X.
\]

On the other hand, composing $g$ with $f^\ell$ and using the relation
$f^\ell\circ\chi_\nu=\psi_\nu$, lower semicontinuity of area and
\eqref{eq:psi-area-C} give, for every $A>0$,
\[
 \int_{\bD_A}(f^\ell\circ g)^*\omega_H\le C.
\]
Letting $A\to\infty$ yields
\[
 \deg\bigl((f^\ell\circ g)^*H\bigr)\le C.
\]
But $(f^\ell)^*H\sim q^\ell H$, and $H$ is ample. Thus
\[
 \deg\bigl((f^\ell\circ g)^*H\bigr)
 =q^\ell\deg(g^*H)\ge q^\ell>C,
\]
a contradiction.  Hence $U$ is taut.
\end{proof}

\section{Projective spaces and toric maps}
\label{sec:ramification-and-toric}

Let \(N\geq 2\) and \(f:\bP^N\to \bP^N\) be a PCF endomorphism of degree \(d\geq 2\) with PCF divisor \(P_f\). We will denote
\begin{equation}\label{eq:pullback-divisor-definitions}
D_0:=P_f,  \quad
  D_{n+1}:=\bigl((f^{n+1})^* D_0\bigr)_{\mathrm{red}}=(f^*D_{n})_{\mathrm{red}},
  \quad
  U_n:=\bP^N\setminus D_n,
\end{equation}
where the subscript \(\mathrm{red}\) means the corresponding reduced divisor. 
Since $f(D_0)\subseteq D_0$, one has
\begin{equation}\label{eq:pullback-chain-inclusion}
D_0\subseteq D_1\subseteq D_2\subseteq\cdots.
\end{equation}

\subsection{Some basic pullback formulas}\label{subsec:ramification}

Recall that the canonical bundle is the line bundle of holomorphic top-degree differential forms; on $\bP^N$, it is $\mathcal O_{\bP^N}(-N-1)$. We let \(R_f\) be the ramification divisor of \(f\) so that the canonical bundle of \(\bP^N\) satisfies
\begin{equation}\label{eq:ramification-canonical-formula}
 K_{\bP^N}=f^*K_{\bP^N}+R_f
\end{equation}
More precisely \(R_f\) is the zero divisor of the determinant of the differential \(\det(df):f^*K_{\bP^N}\to K_{\bP^N}\). It can be expressed as 
\[
R_f=\sum_E \bigl(e_E(f)-1\bigr)E,
\]
where the sum runs over the prime divisors \(E\subset \bP^N\), and \(e_E(f)\) is the ramification index of \(f\) along \(E\).

\begin{lem}
\label{lem:log_pullback}
For every $n\in \bN$, we have
\[
  f^*D_n=R_f+D_{n+1}, \qquad
  K_{\bP^N}+D_{n+1}=f^*(K_{\bP^N}+D_n).
\]
\end{lem}
\begin{proof}
Let $E\subset\bP^N$ be a prime divisor with ramification index $e_E(f)$. 

Assume first that $f(E)\subset D_n$.  As $D_n$ is reduced, \(\operatorname{ord}_E(f^*D_n)=e_E(f)\). 
The coefficient of $E$ in $R_f$ is $e_E(f)-1$, while the coefficient of $E$ in $D_{n+1}$ is one.  Thus both terms \(f^*D_n,R_f+D_{n+1}\) have coefficient $e_E(f)$ along $E$. Assume now that $f(E)\not\subset D_n$.  Then $E\not\subset D_{n+1}$. In particular, \(E\) is not contained in the critical locus of \(f\), so $e_E=1$. Again the coefficient of \(E\) in both \(f^*D_n\) and \(R_f+D_{n+1}\) are the same. Thus, \(f^*D_n=R_f+D_{n+1}\).
  
The second formula of Lemma~\ref{lem:log_pullback} then follows from \eqref{eq:ramification-canonical-formula}.
\end{proof}

\begin{cor}\label{cor:degree-recurrence}
We have 
\[
\deg D_{n+1}-(N+1)=d\bigl(\deg D_n-(N+1)\bigr),
\quad
\deg D_n=N+1+d^n (\deg D_0-N-1).
\]
In particular, \(\deg D_0\geq N+1\). 
\end{cor}

\begin{proof}
Let $H$ be the hyperplane class. Since \(K_{\bP^N}\sim -(N+1)H\) and $f^* H\sim dH$, the two equalities in Corollary~\ref{cor:degree-recurrence} follow from Lemma~\ref{lem:log_pullback}. 
Taking \(n=1\) in the equality we get 
\begin{equation}\label{eq:D1-degree-formula}
d\deg D_0-\deg D_1=(d-1)(N+1).
\end{equation}
Note that $\deg D_0\leq \deg D_{1}$ because $D_0\subset D_{1}$, so \(d\deg D_0-\deg D_1\leq (d-1)\deg D_0\). Then the inequality \(\deg D_0\geq N+1\) follows from \eqref{eq:D1-degree-formula}. 
\end{proof}

\subsection{The minimal degree case}\label{subsec:minimal-degree}
An endomorphism \(f\) of a projective variety \(Y\) is called \emph{polarized} if there exists an ample Cartier divisor \(H\) of \(Y\) such that \(f^*H\sim dH\) for some integer \(d\geq 2\).

A normal variety \(Y\) of dimension \(N\) is \emph{toric} if it contains an algebraic torus \(T:=(\bC^*)^N\) as a dense open subset and if the natural action of \(T\) on itself extends to \(Y\). In this case the complement \(Y\setminus T\) is either empty or a divisor, and \((Y,Y\setminus T)\) is called a \emph{toric pair}. 

In this subsection we will explain how a Theorem of Meng--Zhong from birational geometry leads to a reduction in our study of PCF endomorphisms.
\begin{thm}[{Meng--Zhong~\cite[Th.~1.5]{MengZhong}}]\label{thm:MengZhong}
If a smooth rationally connected variety \(X\) admits a polarized endomorphism \(f\) and a reduced divisor \(D\) such that \(f:X\setminus D\to X\setminus D\) is étale, then \((X,D)\) is a toric pair. 
\end{thm}
We refer to \cite{MengZhong} for stronger statements and to \cite{MYY25} for a variant of Meng--Zhong's theorem.

\begin{cor}\label{cor:projective-space-minimal-toric-pair}
Let \(f:\bP^N\to \bP^N\) be a PCF endomorphism. If the PCF divisor \(P_f\) has degree exactly \(N+1\), then \((\bP^N,P_f)\) is a toric pair. Up to a conjugation by an element of \(\PGL(N+1,\bC)\), the divisor \(P_f\) is the union of \(N+1\) coordinate hyperplanes defined by the equation \(x_0x_1\cdots x_N=0\) where \(x_0,\cdots,x_N\) are homogeneous coordinates, and \(f\) is of the form
\[
f:[x_0:\cdots:x_N]\mapsto [\alpha_0 x_{\sigma(0)}^{d}:\cdots:\alpha_N x_{\sigma(N)}^d]
\]
with \(\alpha_0,\cdots,\alpha_N\in \bC^*\) and \(\sigma\in \mathfrak{S}_{N+1}\).
\end{cor}

\begin{proof}
As \(\deg D_0=N+1\), Corollary~\ref{cor:degree-recurrence} implies that all \(D_i\) are the same. Then Meng--Zhong's theorem (Th.~\ref{thm:MengZhong}) implies that \((\bP^N,P_f)\) is a toric pair. The only toric structure of \(\bP^N\) is up to conjugation the one in the statement.

In homogeneous coordinates already chosen according to the toric structure, we write the map \(f\) as
\[
f:[x_0:\cdots:x_N]\mapsto [F_0:\cdots:F_N]
\]
where the \(F_i\)s are homogeneous polynomials without common factor. The pullback of the coordinate hyperplane \(x_i=0\) is the divisor defined by \(F_i=0\), whose support is contained in \(P_f\). Therefore each \(F_i\) is necessarily a monomial of the form \(\alpha_i x_0^{d_0}\cdots x_N^{d_N}\). 

Now for each index \(j\) consider the point \(e_j\in \bP^N\) whose homogeneous coordinates are all zero except the \(j\)-th one. As \(f\) is an endomorphism, it is well defined at \(e_j\), so at least one index \(i=i_j\) satisfies \(F_i(e_j)\neq 0\). This implies that \(F_i=\alpha_i x_j^d\). Then \(j\mapsto i_j\) is necessarily a permutation. 
\end{proof}

\begin{proof}[Proof of Theorem~\ref{mainthm:Projective-space-pcf}]
By Theorem~\ref{mainthm:rationally-connected}, either Case~\textup{(2)} holds or there is
a reduced divisor $D_\infty$ with
$f^{-1}(D_\infty)=D_\infty$ and $\Crit(f)\subset D_\infty$.  In the latter case
$f$ is etale on $\bP^N\setminus D_\infty$.  From Theorem~\ref{thm:MengZhong} we know that $(\bP^N,D_\infty)$ is a toric pair.  After a
projective change of coordinates,
\[
 D_\infty=\bigcup_{i=0}^N\{x_i=0\}.
\]
Hence $\deg D_\infty = N+1$.

On the other hand, we repeatedly use the fact that 
\[
f(\Crit(f)) \subset f(f^{-1}(D_\infty)) \subset D_\infty
\]
and get $D_0 = P_f \subset D_\infty$. Combining with Corollary~\ref{cor:degree-recurrence} one has
\[
N+1 \le \deg P_f \le \deg D_\infty =N+1.
\]
Now Theorem~\ref{mainthm:Projective-space-pcf} follows from Corollary~\ref{cor:projective-space-minimal-toric-pair}.
\end{proof}

\section{PCF endomorphisms on $\bP^2$}

In this section, we give the proof of Theorem~\ref{mainthm:P2}, which is a refinement of Theorem~\ref{mainthm:Projective-space-pcf} on the projective plane. We also construct explicitly a PCF endomorphism $F$ on $\bP^2$ such that the PCF complement of $F$ is Kobayashi hyperbolic but not hyperbolically embedded into $\mathbb{P}^2$, while $F$ as well as any iterate of $F$ does not preserve a rational pencil.

\subsection{Classification on $\bP^2$}\label{Sect_classify_P2}

Let $f$ be as in Theorem~\ref{mainthm:P2}, and put
$D_n=(f^{-n}P_f)_{\rm red}$ and $U_n=\bP^2\setminus D_n$. We use the non-tautness to produce a pencil with specific property:

\begin{lem}\label{lem:p2-unit-pencil}
Suppose that $U=U_0$ is not taut and that $D_n$ does not stabilize.  Then there
is an irreducible equivariant pencil $\pi:\bP^2\dashrightarrow\bP^1$ with $\pi\circ f=g\circ\pi$ and $g$ is PCF of degree $\deg f$, such that, for
$L=\pi^*\bC(\bP^1)$ and all sufficiently large $n$,
\begin{equation}\label{eq:p2-relative-unit-rank}
\operatorname{rk}\!\left(\frac{\mathcal{O}(U_n)^*}{\mathcal{O}(U_n)^*\cap L^*}\right)\le 1.
\end{equation}
\end{lem}

\begin{proof}
Since $D_n$ does not stabilize, Lemma~\ref{lem:unit-growth} gives
$\operatorname{rk}\bigl(\mathcal{O}(U_n)^*/\bC^*\bigr)\to\infty$. Similar as in Sect.~\ref{unit_map_sect}, we choose $n_0$ so that for every
$n\ge n_0$ the full unit map
\[
 \alpha_n:U_n\longrightarrow T_n=(\bC^*)^{\rho_n}
\]
has $\rho_n>2$.  Write
$Y_n=\overline{\alpha_n(U_n)}^{\rm Zar}$,
$A_n=\Stab(Y_n)^0$, and let $\beta_n:U_n\longrightarrow Z_n$ be the composition with the quotient by $A_n$. Clearly the stabilizer of $Z_n$ is finite.  Lemma~\ref{lem:unit-nondegenerate} gives $\dim Z_n>0$.  If $\dim Z_n=2$,
then $\beta_n$ is generically finite, which implies the tautness of $U$ by Theorem~\ref{tautness-criterion}. A contradiction. Hence $\dim Z_n=1$.  Since
$\dim Y_n\le2$ and the fibers of $Y_n\to Z_n$ are $A_n$-orbits, we also have
$\dim A_n\le1$.

Put $E=\beta_{n_0}^*\bC(Z_{n_0})$ and let $L$ be its relative algebraic closure
in $\bC(\bP^2)$.  For any $n\ge n_0$, if the map
$(\beta_{n_0},\beta_n):U_n\to Z_{n_0}\times Z_n$ had two-dimensional image,
it would be generically finite onto that image.  The stabilizer of the Zariski closure of the image is finite.  This again contradicts Theorem~\ref{tautness-criterion}.  Therefore the
composite of $E$ and $\beta_n^*\bC(Z_n)$ has transcendence degree one, and
relative algebraic closedness gives $\beta_n^*\bC(Z_n)\subset L$.

Applying the same argument to
$(\beta_{n_0},\beta_{n_0}\circ f)$ on $U_{n_0+1}$ gives $f^*E\subset L$.
If $a\in L$, then $f^*a$ is algebraic over $f^*E\subset L$ and belongs to
$\bC(\bP^2)$, so the relative algebraic closedness of $L$ gives
$f^*a\in L$.  Thus $f^*L\subset L$.

The invariant-field argument and the Chow-reduction construction in the proof of Proposition~\ref{prop:pcf-rational-fibration} therefore produces an irreducible equivariant pencil $\pi$. Its base is
rational because it is dominated by $\bP^2$. Proposition~\ref{prop:pcf-descent} shows that the induced $g$ is PCF.  Polarization descent gives
$\deg g=\deg f$.

Finally, we identify $\mathcal{O}(U_n)^*/\bC^*$ with the character lattice $X^*(T_n)$ of the torus $T_n$. Note that every character trivial on $A_n$ descends to $T_n/A_n$, and thus its pullback
belongs to $\beta_n^*\bC(Z_n)\subset L$. If we denote the quotient group $\left.\mathcal{O}(U_n)^*\middle/ (\mathcal{O}(U_n)^*\cap L^*)\right.$ as $\mathcal{Q}_1$ and denote the quotient group $\left.X^*(T_n) \middle/ X^*(T_n/A_n)\right.$ by $\mathcal{Q}_2$, then $\mathcal{Q}_2 \subset X^*(A_n)$ and one has the following diagram
\[
\begin{tikzcd}
0 \arrow[r] & X^*(T_n/A_n) \arrow[r] \arrow[d] 
  & X^*(T_n) \arrow[r] \arrow[d, "\cong"] 
  & \mathcal{Q}_2 \arrow[r] \arrow[d, two heads] 
  & 0 \\
0 \arrow[r] 
  & \left.(\mathcal{O}(U_n)^*\cap L^*)\middle/ \bC^*\right. \arrow[r] 
  & \left.\mathcal{O}(U_n)^*\middle/ \bC^*\right. \arrow[r] 
  & \mathcal{Q}_1 \arrow[r] 
  & 0.
\end{tikzcd}
\]
Therefore the quotient $\mathcal{Q}_1$ in
\eqref{eq:p2-relative-unit-rank} is a quotient of $\mathcal{Q}_2$, whose rank is at most $\dim A_n\le1$.
\end{proof}

We will also use the following classification result.

\begin{thm}[Dabija--Jonsson]\label{thm:DJ}
An endomorphism of $\bP^2$ of degree $D>1$ preserving an irreducible
pencil has, in suitable projective coordinates, one of two forms
\cite[Main Theorem]{DabijaJonsson2}:
\begin{enumerate}[label=\textup{(\roman*)}]
\item \emph{Elementary pencil:}
$\pi=[x:y]$ and $f=[P(x,y):Q(x,y):R(x,y,z)]$, with homogeneous
polynomials of degree $D$ and $g=[P:Q]$.
\item \emph{Binomial pencil:}
$\pi=[x^hy^{k-h}:z^k]$, where $0<h<k$ and $\gcd(h,k)=1$, and
\small
\[
 f[x:y:z]=[x^d:y^d:R]\quad\hbox{or}\quad f[x:y:z]=[y^d:x^d:R],\qquad
 R=z^{d-k\ell}\prod_{i=1}^{\ell}(z^k+c_i x^hy^{k-h}),
\]
\normalsize
where $0\le\ell\le d/k$ and $c_i\in\bC^*$.
\end{enumerate}
In the elementary case, the vertical ramification divisor satisfies
$R_f^{\mathcal P}=\pi^*R_g$, where $R_f^{\mathcal P}$ is the part of the ramification locus of $f$ supported on fibers of $\pi$
\cite[Lemma~1.2]{DabijaJonsson2}.
\end{thm}

\begin{proof}[Proof of Theorem~\ref{mainthm:P2}]
Assume $U_0=\bP^2\setminus P_f$ is not taut.  If the divisors $D_n$ stabilize as in
Corollary~\ref{cor:divisorial-tower}, we obtain a totally invariant divisor $D_\infty$ containing the
critical divisor. Using the same argument as in the proof of Theorem~\ref{mainthm:Projective-space-pcf}, one can show that, after a projective change of
coordinates, $f:\,[x:y:z] \mapsto [c_0x^d:c_1y^d:c_2z^d]$, giving the Case~\textup{(1)}.

Assume henceforth that $D_n$ does not stabilize and choose the pencil of
Lemma~\ref{lem:p2-unit-pencil}.  We may replace $f$ by positive iterates and continue to denote the resulting map by $f$, with degree denoted by $d$. Note that $P_{f^a}=P_f$, and the new boundary tower is some subsequence $D_{an}$, so \eqref{eq:p2-relative-unit-rank} persists.

Consider the PCF base map $g:\, \bP^1 \to \bP^1$ with $\#P_g\ge2$. If $P_g=\{a,b\}$, then one can use the Riemann--Hurwitz formula and the forward invariance of
$P_g$ to show that $g^{-1}(P_g)=P_g$.\\[.1cm]

\paragraph{\emph{Binomial case.}}
After squaring if necessary, we use the form $f:\,[x:y:z] \mapsto [x^d:y^d:R]$ and put
\[
 t=\frac{x^hy^{k-h}}{z^k},
 \qquad
 H(t)=\prod_i(1+c_it).
\]
Then $R=z^dH(t)$ and $g(t)=t^d/H(t)^k$. Since $\operatorname{gcd}(h,k-h)=1$, one can choose integers $a,b$ with $hb-(k-h)a=1$ and set $u=(x/z)^a(y/z)^b$. Then the coprime condition guarantees that $(t,u)$ give the coordinates of a 2-dimensional torus. Direct computation gives
\begin{equation}\label{f_toric_coord}
 f:\,(t,u) \mapsto \bigl(g(t), u^d \cdot H(t)^{-a-b}\bigr).
\end{equation}
If $\#P_g\ge3$, let $S=\bP^1\times\bP^1$ with coordinates $(t,u)$,
let $\rho$ be the first projection, and let
$\tau:S\dashrightarrow\bP^2$ be the birational map induced by the inverse of the torus coordinates.  The formula \eqref{f_toric_coord} defines a rational
self-map $G$ of $S$ and gives $\tau\circ G=f\circ\tau$ and $\rho\circ G=g\circ\rho$.
Thus Case~\textup{(2)} holds.

If $\#P_g=2$, then $0$ is a totally ramified fixed point of $g$ and the other postcritical point is $\infty$. We will show that $H(t) \equiv 1$. Otherwise, $H$ is nonconstant,  and every zero of $H$ is a pole of local degree at least $k\ge2$. Note that complete invariance of
$\{0,\infty\}$ would make $\infty$ totally invariant, contradicting the
finite pole.  Hence $H \equiv 1$ and, after at most one further squaring,
$f=[x^d:y^d:z^d]$, which is Case~\textup{(1)}.\\[.1cm]

\paragraph{\emph{Elementary case.}}
Write $\pi[x:y:z]=[x:y]$ and $f[x:y:z]=[P(x,y):Q(x,y):R(x,y,z)]$.
Clearly, the base point $p=[0:0:1]$ is totally invariant.  Note that the vertical components of $D_n$ are exactly the pencil
lines above $g^{-n}(P_g)$, while inverse images of horizontal components are
horizontal.  Fix a vertical line $L_0=\{l=0\}\subset P_f$.  If
$C=\{F_C=0\}$ is any horizontal irreducible component of some $D_n$, then
\[
 u_C=\frac{F_C}{l^{\deg C}}
\]
is a unit on $U_n$.  Two distinct horizontal components would give two such
units independent modulo $L^*=\pi^*\bC(\bP^1)^*$, and thus contradict to \eqref{eq:p2-relative-unit-rank} after passing to a common level $n$.  Hence the entire boundary tower
contains at most one horizontal irreducible component.

Blow up the base point,
\[
\tau:S:=\operatorname{Bl}_p\bP^2\longrightarrow\bP^2,
 \qquad
 \rho:S\to\bP^1,
\]
and let $G:S\dashrightarrow S$ be the rational lift of $f$.  The exceptional
divisor is a section of $\rho$ and is totally invariant. For large $n$ and a
general fiber $S_t$ of $\rho$, both $S_t\setminus \tau^*D_{n+1}$ and
$S_{g(t)}\setminus \tau^*D_n$ are projective lines with the same number $r$ of
punctures.
The restriction of $G$ to a general fiber $S_t$ has degree $d$, since
$R(0,0,1)\ne0$.  By Lemma~\ref{lem:etale-tower}, it is an unbranched covering
between these punctured projective lines (take $t\notin\Crit(g)$).
Hence $2-r=d(2-r)$, which implies $r=2$. Besides the exceptional section there is therefore
one horizontal component, which meets a general fiber once. After shrinking
the base it is a second section.

Now we choose a rational fiber coordinate $u$ sending the horizontal section to $0$
and the exceptional section to $\infty$.  Both sections are totally invariant
on the generic fiber.  Thus the degree-$d$ map on the generic $\bP^1$ has
$0$ and $\infty$ as its only preimages of $0$ and $\infty$, respectively.
Consequently, we can write $G(t,u)=\bigl(g(t),a(t)u^D\bigr)$ for some $a(t)\in\bC(t)^*$. This gives the Case~\textup{(2)}.
\end{proof}

\subsection{An example}

In this section, we study the following explicit endomorphism
\begin{equation}\label{eq:cheb-F}
 F[X:Y:Z]=[X^2-2YZ:Y^2-2XZ:Z^2]
\end{equation}
on $\bP^2$.
\begin{prop}\label{hyper-emb-counterexample}
The following properties hold for $F$:
\begin{enumerate}
    \item $F$ is PCF;
    \item $F$ as well as any iterate of $F$ does not preserve a rational pencil;
    \item The PCF complement $U_F:= \bP^2 \setminus P_F$ is taut;
    \item However, $U_F$ is \emph{not hyperbolically embedded into $\bP^2$}.
\end{enumerate}
\end{prop}
The rest of this section is devoted to the proof of Proposition~\ref{hyper-emb-counterexample}.\\[.1cm]

\paragraph{\emph{$F$ is PCF}}
By computing the determinant of the homogeneous Jacobian of $F$, $\det J_F=8Z(XY-Z^2)$,
we obtain that
\[
 \Crit(F)=H_\infty\cup C,
 \qquad H_\infty=\{Z=0\},\qquad C=\{XY=Z^2\}.
\]
Restrict $F$ to the affine char $\{Z=1\}$. Denote by $x:=X/Z$ and $y:=Y/Z$ the affine coordinates. Then $F(x,y)=(x^2-2y,y^2-2x)$ and $C=\{xy=1\}$.

Now we consider a subtorus $T=\{(t_1,t_2,t_3)\in(\bC^*)^3:t_1t_2t_3=1\}\simeq(\bC^*)^2$ of $(\bC^*)^3$. Consider the map
\[
 \Pi:\, T \to \bC^2, \qquad  (t_1,t_2,t_3) \mapsto
 \left(t_1+t_2+t_3,\;t_1^{-1}+t_2^{-1}+t_3^{-1}\right).
\]
Since $t_1t_2t_3=1$, the two components of $\Pi$ are the first two elementary symmetric functions. Thus $\Pi:T\to\bC^2$ is induced by the finite quotient of $T$ by $S_3$.

Now we consider a self-map $M_2$ on the torus $T$
\[
M_2(t_1,t_2,t_3):=(t_1^2,t_2^2,t_3^2).
\]

Note that $M_2$ is \emph{semiconjugate with $F$ via $\Pi$}: By direct computation one obtains that $\sum_i t_i^2=x^2-2y$ and $\sum_i t_i^{-2}=y^2-2x$. Hence we have
\small
\[
\Pi \circ M_2(t_1,t_2,t_3) = \left( \sum_i t_i^2, \sum_i t_i^{-2}\right) = (x^2-2y, y^2-2x) = F(x,y) = F \circ \Pi(t_1,t_2,t_3).
\]
\normalsize
Let $D\subset\bP^2$ be the projective completion of the image curve of $\{t_1=t_2\}\subset T$ under the map $\Pi$.  The affine parametrization
$(x,y)=(s,s^{-1})$ of $C$ gives
\[
 F(s,s^{-1})=(s^2-2s^{-1},s^{-2}-2s)
            =\Pi(t,t,t^{-2}),\quad \textrm{with}\,\,  t=-s^{-1}.
\]
Hence $F(C)=D$. One can also check that $F(D)=D$: Since $M_2(t,t,t^{-2}) = (t^2,t^2,t^{-4})$, from the semiconjugacy between $M_2$ and $F$ we can obtain that $F(D)=D$. Clearly $F(H_\infty)=H_\infty$.
Therefore
\begin{equation}\label{eq:PF-counter}
 P_F=D\cup H_\infty.
\end{equation}
In particular $F$ is PCF.\\[.1cm]

\paragraph{\emph{$U_F$ is hyperbolic}}
Now we study the PCF complement $U_F:= \bP^2 \setminus P_F$. Consider a Zariski open of $T$:
\[
 T^\circ:=\{(t_1,t_2,t_3)\in T:t_i\ne t_j\text{ for }i\ne j\}.
\]
The $S_3$-action is free on $T^\circ$, and its quotient is the complement of the
 discriminant. Thus
\[
 \Pi:T^\circ\longrightarrow U_F:=\bP^2\setminus P_F=\bC^2\setminus D
\]
is a finite etale covering.

There is also a convenient projective quotient picture.  Projectivization gives
\[
 q:T^\circ\longrightarrow \bP^2,
 \qquad
 q(t_1,t_2,t_3)=[t_1:t_2:t_3].
\]
Its image is the complement of the six-line arrangement
\[
 \mathcal A=\{XYZ(X-Y)(Y-Z)(Z-X)=0\}\subset\bP^2,
\]
namely the classical Ceva arrangement, equivalently the complete quadrilateral
 determined by the four triple points $[1:0:0]$, $[0:1:0]$, $[0:0:1]$ and $[1:1:1]$. See the following Figure~\ref{fig:ceva-arrangement} for this line arrangement.

\begin{figure}[t]
\centering
\begin{tikzpicture}[x=0.013cm,y=-0.013cm,
                    line width=0.5pt,line cap=round]
  \coordinate (A) at (437.08,93.42);
  \coordinate (B) at (224.27,330.54);
  \coordinate (C) at (697.81,379.75);
  \coordinate (O) at (407.82,220.22);

  \draw ($(B)!-0.275289!(A)$) -- ($(B)!1.074053!(A)$);
  \draw ($(A)!-0.128971!(C)$) -- ($(A)!1.137875!(C)$);
  \draw ($(B)!-0.218029!(C)$) -- ($(B)!1.048943!(C)$);

  \draw ($(A)!-0.549003!(O)$) -- ($(A)!2.330767!(O)$);
  \draw ($(B)!-0.153535!(O)$) -- ($(B)!2.060755!(O)$);
  \draw ($(O)!-0.412182!(C)$) -- ($(O)!1.282300!(C)$);
\end{tikzpicture}
\caption{$W$ as complement of line arrangement}
\label{fig:ceva-arrangement}
\end{figure}
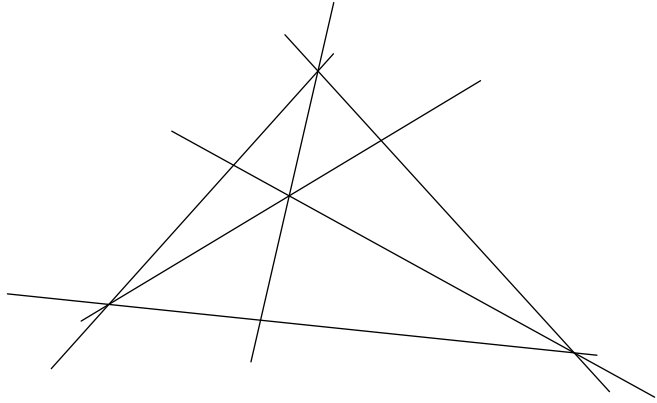

On the affine chart $Z=1$ with coordinates $r$, $s$, we obtain the domain
\[
 W:=\{(r,s)\in(\bC^*)^2:r\ne1,\ s\ne1,\ r\ne s\},
\]
and the map $q$ is written there as
\[
 \rho:T^\circ\longrightarrow W,
 \qquad
 \rho(t)=\left(\frac{t_1}{t_3},\frac{t_2}{t_3}\right).
\]
For fixed $(r,s)$ one has $t_3^3=(rs)^{-1}$, so $\rho$ is a three-sheeted
unramified covering. Denote $\Omega=\bC^*\setminus\{1\}\simeq\bP^1\setminus\{0,1,\infty\}$. Then $W\subset\Omega^2$, and hence $W$ is Kobayashi hyperbolic. Since hyperbolicity
is preserved under finite etale coverings,
we know that $T^\circ$ and $U_F$ are also Kobayashi hyperbolic.\\[.1cm]

\paragraph{\emph{$F$ does not preserve a rational pencil}}
Now we consider the iterate $h=F^m$, of degree $d \ge 2^m$ ($m \ge 1$). Note that iteration does not change the
postcritical set: set-theoretically one has $\Crit(F^m)=\bigcup_{i=0}^{m-1}F^{-i}(\Crit F)$,
and therefore
\[
 P_{F^m}=\bigcup_{j\ge1}\bigcup_{i=0}^{m-1}
 F^{mj-i}(\Crit F)=P_F.
\]
Thus
\begin{equation}\label{eq:PF-iterate}
 P_h=P_F=D\cup H_\infty.
\end{equation}

Suppose that $h$ preserves some rational pencil. Similar as we did in Sect.~\ref{Sect_classify_P2}, one takes the relative
algebraic closure of the factor field and obtains a one-dimensional $h$-invariant
function field.  Its normalized base is rational (it is dominated by $\bP^2$),
so it defines an $h$-invariant irreducible pencil
$\pi:\bP^2\dashrightarrow\bP^1$.  By
Theorem~\ref{thm:DJ}, the pencil is either elementary or binomial.

Assume first that $\pi$ is elementary.  In coordinates with
$\pi[x:y:z]=[x:y]$, the Dabija--Jonsson ramification identity is $R_h^{\mathcal P}=\pi^*R_g$, where $g:\bP^1\to\bP^1$ is the induced degree-$d$ map.  Then $g$ is PCF by Proposition~\ref{prop:pcf-descent}.  Moreover $\#P_g\ge2$ by the Riemann--Hurwitz formula.  For every $q\in P_g$ choose a critical point
$c$ of $g$ and $j\ge1$ with $q=g^j(c)$.  The fiber $L_c=\pi^{-1}(c)$ is a
critical line of $h$, and $h^j(L_c)=L_q$.  Hence every $L_q$ is a line
component of $P_h$.  Thus $P_h$ has at least two distinct line components,
contradicting \eqref{eq:PF-iterate}, whose unique line component is
$H_\infty$.

Assume next that the pencil is binomial.  In the Dabija--Jonsson normal form,
the reducible member $x^hy^{k-h}=0$ consists of two distinct lines $L_1,L_2$,
and $h$ fixes or interchanges them.  Consequently both lines are totally
invariant under $h^2$.  Since $(h^2)^*L_i$ has degree $d^2$ and is supported on
$L_i$, one has
\[
 (h^2)^*L_i=d^2L_i.
\]
Thus $L_i$ is a ramification component of $h^2$. Since it is also invariant,
$L_i\subset P_{h^2}$.  But $P_{h^2}=P_F$ by \eqref{eq:PF-iterate}, so $P_F$
would again have two distinct line components, a contradiction. Hence no positive iterate of $F$ preserves a rational pencil.

Recall that $U_F$ is Kobayashi hyperbolic. Thus $F$ cannot be a monomial power map. By Theorem~\ref{mainthm:Projective-space-pcf} we conclude that $U_F$ is \emph{taut}.\\[.1cm]

\paragraph{\emph{$U_F$ is not hyperbolically embedded into $\bP^2$}}
Now we construct two sequences of points on $T^\circ$, $\{a_n\}$ and $\{b_n\}$, and map them to sequences of points on $U_F$ via $\Pi$. Then we will show that the Kobayashi distance between $\Pi(a_n)$ and $\Pi b_n$ decreases to zero as $n \to \infty$.

Let $R_n=e^{n^2}$, $\epsilon_n=\frac1n$. We put $a_n=(R_n^{1+\epsilon_n},R_n^{-1},R_n^{-\epsilon_n})$ and $b_n=(R_n,R_n^{-1-\epsilon_n},R_n^{\epsilon_n})$. Then $a_n,b_n \in T^\circ$, and it is easy to check that
\[
 \Pi(a_n)\longrightarrow A=[1:0:0],\qquad
 \Pi(b_n)\longrightarrow B=[0:1:0].
\]
On the other hand, one has $\rho(a_n)=\bigl(R_n^{1+2\epsilon_n},R_n^{-1+\epsilon_n}\bigr)$, $\rho(b_n)=\bigl(R_n^{1-\epsilon_n},R_n^{-1-2\epsilon_n}\bigr)$.
Both points eventually lie in
\[
 Q=Q_1\times Q_2:=\{|r|>2\}\times\{0<|s|<1/2\}\subset W.
\]
Each $Q_i$ is biholomorphic to a punctured disc: $Q_1\simeq\Delta^*$ via
$r\mapsto 2/r$, and $Q_2\simeq\Delta^*$ via $s\mapsto 2s$. Hence, we can write the
Kobayashi--Royden metric as
\[
 F_{Q_1}(r;\xi)=\frac{|\xi|}{2|r|\log(|r|/2)},
 \qquad
 F_{Q_2}(s;\eta)=\frac{|\eta|}{2|s|\log(1/(2|s|))}.
\]
We choose a special path from $\rho(a_n)$ to $\rho(b_n)$: Join $\rho(a_n)$ to $\rho(b_n)$ inside $Q$ by first varying the $r$-coordinate
along the positive real axis and then varying the $s$-coordinate. 
The first
segment has Kobayashi length
\[
 L_n^{(r)}=
 \int_{R_n^{1-\epsilon_n}}^{R_n^{1+2\epsilon_n}}
 \frac{dr}{2r\log(r/2)}
 =\frac12\log
 \frac{(1+2\epsilon_n)\log R_n-\log2}
      {(1-\epsilon_n)\log R_n-\log2},
\]
which tends to $0$ because $\epsilon_n\to0$.  Similarly the second segment has
length
\[
 L_n^{(s)}=
 \int_{R_n^{-1-2\epsilon_n}}^{R_n^{-1+\epsilon_n}}
 \frac{ds}{2s\log(1/(2s))}
 =\frac12\log
 \frac{(1+2\epsilon_n)\log R_n-\log2}
      {(1-\epsilon_n)\log R_n-\log2}
 \longrightarrow0.
\]
Thus we have found a path $\gamma_n\subset Q$, from $\rho(a_n)$ to $\rho(b_n)$,
whose Kobayashi length tends to $0$ as $n \to \infty$.

Now we lift $\gamma_n$ to a path on $T^\circ$, connecting $a_n$ and $b_n$. Recall that $\rho:\, T^\circ \to W \subset (\bC^*)^2$ is a three-sheeted
unramified covering, which is locally defined by $t_3^3=(rs)^{-1}$. So we choose the positive real branch $t_3=(rs)^{-1/3}$, starting at $a_n$.  At the endpoint this
gives $t_3=R_n^{\epsilon_n}$ and hence exactly the point $b_n$. By the definition of the Kobayashi--Royden metric, we have
\[
 d_{T^\circ}(a_n,b_n)\le \ell_W(\gamma_n)\longrightarrow 0.
\]
Then the distance-decreasing property of $\Pi$ yields $d_{U_F}(\Pi(a_n),\Pi(b_n))\longrightarrow 0$.

However, the limits $A\ne B$ are distinct points of the boundary $H_\infty\subset P_F$.
Hence the Kobayashi distance on $U_F$ does not separate these two
boundary points, and therefore $U_F\hookrightarrow\bP^2$ is not a hyperbolic embedding.

\section*{Acknowledgement}
Our collaboration began during the participation of the second author to the conference "birational geometry and moduli spaces" organised by the first author in January 2026 at Xiamen University, though the work presented in this paper was mostly done later. We thank Xiamen University and other organisers for hospitality and organisation.

R.-R. Sun is supported by an NSFC grant (No. 12501059) and a grant from Xiamen University (No. X2450215). S.-Y. Zhao acknowledges partial support from the French National Research Agency under the projects GAG (ANR-24-CE40-3526-01) and DynAtrois (ANR-24-CE40-1163).
\section*{AI usage}
We developed the main framework of the materials presented in this paper, and used LLM models to draft the paper.

\bibliographystyle{amsplain}
\bibliography{biblio}

@article{KiernanKobayashi1973,
 author = {Kiernan, Peter and Kobayashi, Shoshichi},
 title = {Holomorphic mappings into projective space with lacunary hyperplanes},
 journal = {Nagoya Math. J.},
 volume = {50},
 pages = {199--216},
 year = {1973},
 doi = {10.1017/S0027763000015646}
}

@book{Lazarsfeld2004,
  author    = {Robert Lazarsfeld},
  title     = {Positivity in Algebraic Geometry I:
               Classical Setting: Line Bundles and Linear Series},
  series    = {Ergebnisse der Mathematik und ihrer Grenzgebiete. 3. Folge},
  volume    = {48},
  publisher = {Springer-Verlag},
  address   = {Berlin},
  year      = {2004}
}

@article{Bloch1926,
 author = {Bloch, Andr{\'e}},
 title = {Sur les syst{\`e}mes de fonctions uniformes satisfaisant {\`a}
          l'{\'e}quation d'une vari{\'e}t{\'e} alg{\'e}brique dont
          l'irr{\'e}gularit{\'e} d{\'e}passe la dimension},
 journal = {J. Math. Pures Appl. (9)},
 volume = {5},
 pages = {19--66},
 year = {1926}
}

@article{Ochiai1977,
 author = {Ochiai, Takushiro},
 title = {On holomorphic curves in algebraic varieties with ample irregularity},
 journal = {Invent. Math.},
 volume = {43},
 pages = {83--96},
 year = {1977}
}

@article{Brody1978,
 author = {Brody, Robert},
 title = {Compact manifolds and hyperbolicity},
 journal = {Trans. Amer. Math. Soc.},
 volume = {235},
 pages = {213--219},
 year = {1978},
 doi = {10.1090/S0002-9947-1978-0470252-3}
}

@article {FujinoQuasiAlbCorrection,
    AUTHOR = {Fujino, Osamu},
     TITLE = {Correction: {O}n quasi-{A}lbanese maps},
   JOURNAL = {Boll. Unione Mat. Ital.},
  FJOURNAL = {Bollettino dell'Unione Matematica Italiana},
    VOLUME = {18},
      YEAR = {2025},
    NUMBER = {4},
     PAGES = {1069--1071},
      ISSN = {1972-6724,2198-2759},
   MRCLASS = {14E05 (14E30 14L40 32M05)},
  MRNUMBER = {4985408},
       DOI = {10.1007/s40574-025-00481-8},
       URL = {https://doi-org.proxy3.library.mcgill.ca/10.1007/s40574-025-00481-8},
}

@article {FujinoQuasiAlb,
    AUTHOR = {Fujino, Osamu},
     TITLE = {On quasi-{A}lbanese maps},
   JOURNAL = {Boll. Unione Mat. Ital.},
  FJOURNAL = {Bollettino dell'Unione Matematica Italiana},
    VOLUME = {18},
      YEAR = {2025},
    NUMBER = {4},
     PAGES = {1027--1067},
      ISSN = {1972-6724,2198-2759},
   MRCLASS = {14E05 (14E30 14L40 32M05)},
  MRNUMBER = {4985407},
       DOI = {10.1007/s40574-024-00448-1},
       URL = {https://doi-org.proxy3.library.mcgill.ca/10.1007/s40574-024-00448-1},
}

@article{Abramovich1994,
 author = {Abramovich, Dan},
 title = {Subvarieties of semiabelian varieties},
 journal = {Compos. Math.},
 volume = {90},
 number = {1},
 pages = {37--52},
 year = {1994}
}

@article{Astorg,
 author = {Astorg, Matthieu},
 title = {Dynamics of post-critically finite maps in higher dimension},
 journal = {Ergodic Theory Dyn. Syst.},
 volume = {40},
 number = {2},
 pages = {289--308},
 year = {2020}
}

@article{Crass2007Klein,
  author        = {Crass, Scott},
  title         = {Solving the heptic by iteration in two dimensions: Geometry and dynamics under {Klein}'s group of order 168},
  journal       = {Journal of Modern Dynamics},
  volume        = {1},
  number        = {2},
  pages         = {175--203},
  year          = {2007}
}

@article{Crass2011Valentiner,
  author        = {Crass, Scott},
  title         = {New light on solving the sextic by iteration: An algorithm using reliable dynamics},
  journal       = {Journal of Modern Dynamics},
  volume        = {5},
  number        = {2},
  pages         = {397--408},
  year          = {2011}
}

@article{DabijaJonsson,
 author = {Dabija, Marius and Jonsson, Mattias},
 title = {Algebraic webs invariant under endomorphisms},
 journal = {Publ. Mat., Barc.},
 issn = {0214-1493},
 volume = {54},
 number = {1},
 pages = {137--148},
 year = {2010}
}

@article {DabijaJonsson2,
    AUTHOR = {Dabija, Marius and Jonsson, Mattias},
     TITLE = {Endomorphisms of the plane preserving a pencil of curves},
   JOURNAL = {Internat. J. Math.},
  FJOURNAL = {International Journal of Mathematics},
    VOLUME = {19},
      YEAR = {2008},
    NUMBER = {2},
     PAGES = {217--221},
      ISSN = {0129-167X,1793-6519},
   MRCLASS = {32H50 (37F10)},
  MRNUMBER = {2384901},
MRREVIEWER = {Charles\ Favre},
       DOI = {10.1142/S0129167X08004637},
       URL = {https://doi-org.proxy3.library.mcgill.ca/10.1142/S0129167X08004637},
}

@book{DebarreBook,
 author = {Debarre, Olivier},
 title = {Higher-dimensional algebraic geometry},
 series = {Universitext},
 issn = {0172-5939},
 isbn = {0-387-95227-6},
 year = {2001},
 publisher = {New York, NY: Springer}
}

@article{DKMpcfset,
 author = {DeMarco, Laura G. and Koch, Sarah C. and McMullen, Curtis T.},
 title = {On the postcritical set of a rational map},
 journal = {Math. Ann.},
 issn = {0025-5831},
 volume = {377},
 number = {1-2},
 pages = {1--18},
 year = {2020}
}

@incollection{FornaessSibonyCF,
 author = {Forn{\ae}ss, John Erik and Sibony, Nessim},
 title = {Critically finite rational maps on {{\({\mathbb{P}{}}^ 2\)}}},
 booktitle = {The Madison symposium on complex analysis. Proceedings of the symposium, held June 2-7, 1991, at the University of Wisconsin-Madison, Madison, WI, USA},
 pages = {245--260},
 year = {1992},
 publisher = {Providence, RI: American Mathematical Society}
}

@incollection{FornaessSibonyKobayashi,
 author = {Fornaess, John Erik and Sibony, Nessim},
 title = {Complex dynamics in higher dimension. {I}},
 booktitle = {Complex analytic methods in dynamical systems. Proceedings of the congress held at Instituto de Matem\'atica Pura e Aplicada, IMPA, Rio de Janeiro, Brazil, January 1992},
 pages = {201--231},
 year = {1994},
 publisher = {Paris: Soci{\'e}t{\'e} Math{\'e}matique de France}
}

@article{GTV,
 author = {Gauthier, Thomas and Taflin, Johan and Vigny, Gabriel},
 title = {Sparsity of postcritically finite maps of {{\(\mathbb{P}^k\)}} and beyond: a complex analytic approach},
 journal = {Publ. Math., Inst. Hautes {\'E}tud. Sci.},
 volume = {143},
 pages = {1--96},
 year = {2026}
}

@article{EGAIV,
  author  = {Alexander Grothendieck},
  title   = {Éléments de géométrie algébrique {IV}. Étude locale des schémas et des morphismes de schémas},
  journal = {Publ. Math. Inst. Hautes \'Etudes Sci.},
  year    = {1964},
  note    = {Vol. 20 (1964), 5--259; Vol. 24 (1965), 5--231; Vol. 28 (1966), 5--255; Vol. 32 (1967), 5--361},
  pages = {}
}

@book{SGA1,
 editor = {Grothendieck, A.},
 title = {S{\'e}minaire de g{\'e}om{\'e}trie alg{\'e}brique du {Bois} {Marie} 1960-61. {Rev{\^e}tements} {\'e}tales et groupe fondamental ({SGA} 1). {Un} s{\'e}minaire dirig{\'e} par {Alexander} {Grothendieck}. {Augment{\'e}} de deux expos{\'e}s de {M}. {Raynaud}.},
 edition = {{\'E}dition recompos{\'e}e et annot{\'e}e du original publi{\'e} en 1971 par {Springer}},
 fseries = {Documents Math{\'e}matiques},
 series = {Doc. Math. (SMF)},
 issn = {1629-4939},
 volume = {3},
 isbn = {2-85629-141-4},
 year = {2003},
 publisher = {Paris: Soci{\'e}t{\'e} Math{\'e}matique de France}
}

@article{Ji,
 author = {Ji, Zhuchao},
 title = {Structure of {Julia} sets for post-critically finite endomorphisms on {{\(\mathbb{P}^2\)}}},
 journal = {Math. Ann.},
 issn = {0025-5831},
 volume = {384},
 number = {1-2},
 pages = {857--880},
 year = {2022}
}

@article{Jonsson,
 author = {Jonsson, Mattias},
 title = {Some properties of 2-critically finite holomorphic maps {{\(\mathbb{P}^2\)}}},
 journal = {Ergodic Theory Dyn. Syst.},
 volume = {18},
 number = {1},
 pages = {171--187},
 year = {1998}
}

@article{Kawamata1980,
 author = {Kawamata, Yujiro},
 title = {On {Bloch}'s conjecture},
 journal = {Invent. Math.},
 volume = {57},
 number = {1},
 pages = {97--100},
 year = {1980}
}

@article{Kiernan,
 author = {Kiernan, Peter},
 title = {On the relations between taut, tight and hyperbolic manifolds},
 journal = {Bull. Am. Math. Soc.},
 volume = {76},
 pages = {49--51},
 year = {1970}
}

@book{Kobayashi1998,
 author = {Kobayashi, Shoshichi},
 title = {Hyperbolic Complex Spaces},
 series = {Grundlehren der mathematischen Wissenschaften},
 volume = {318},
 publisher = {Springer-Verlag},
 address = {Berlin},
 year = {1998}
}

@article{Koch,
 author = {Koch, Sarah},
 title = {Teichm{\"u}ller theory and critically finite endomorphisms},
 journal = {Adv. Math.},
 issn = {0001-8708},
 volume = {248},
 pages = {573--617},
 year = {2013}
}

@misc{KollarXu,
 author = {J{\'a}nos Koll{\'a}r and Chenyang Xu},
 title = {Fano {Varieties} with {Large} {Degree} {Endomorphisms}},
 year = {2009},
 howpublished = {Preprint, {arXiv}:0901.1692 [math.{AG}] (2009)},
 url = {https://arxiv.org/abs/0901.1692},
 arXiv = {arXiv:0901.1692}
}

@article{MengZhong,
 author = {Meng, Sheng and Zhong, Guolei},
 title = {Rigidity of rationally connected smooth projective varieties from dynamical viewpoints},
 journal = {Math. Res. Lett.},
 volume = {30},
 number = {2},
 pages = {589--610},
 year = {2023}
}

@article{MengZhangPolarized,
 author = {Meng, Sheng and Zhang, De-Qi},
 title = {Building blocks of polarized endomorphisms of normal projective varieties},
 journal = {Adv. Math.},
 volume = {325},
 pages = {243--273},
 year = {2018}
}

@article{MengZhang,
 author = {Meng, Sheng and Zhang, De-Qi},
 title = {Kawaguchi-{Silverman} conjecture for certain surjective endomorphisms},
 journal = {Doc. Math.},
 volume = {27},
 pages = {1605--1642},
 year = {2022}
}

@article{Nakayama,
 author = {Nakayama, Noboru},
 title = {Intersection sheaves over normal schemes},
 journal = {J. Math. Soc. Japan},
 volume = {62},
 number = {2},
 pages = {487--595},
 year = {2010}
}

@article{MYY25,
 author = {Moraga, Joaqu{\'{\i}}n and Yeong, Wern and Y{\'a}{\~n}ez, Jos{\'e} Ignacio},
 title = {Polarized endomorphisms of {Fano} varieties with complements},
 journal = {Int. Math. Res. Not.},
 volume = {2025},
 number = {8},
 pages = {18},
 year = {2025}
}

@article{RongFatou,
 author = {Rong, Feng},
 title = {The {Fatou} set for critically finite maps},
 journal = {Proc. Am. Math. Soc.},
 volume = {136},
 number = {10},
 pages = {3621--3625},
 year = {2008}
}

@article{Ueda,
 author = {Ueda, Tetsuo},
 title = {Critical orbits of holomorphic maps on projective spaces},
 journal = {J. Geom. Anal.},
 volume = {8},
 number = {2},
 pages = {319--334},
 year = {1998}
}

@article{Van2,
 author = {Le, Van Tu},
 title = {Periodic points of weakly post-critically finite all the way down maps},
 journal = {Ann. Mat. Pura Appl. (4)},
 volume = {202},
 number = {3},
 pages = {1187--1195},
 year = {2023}
}

@article{Van1,
 author = {Le, Van Tu},
 title = {Periodic points of post-critically algebraic holomorphic endomorphisms},
 journal = {Ergodic Theory Dyn. Syst.},
 volume = {42},
 number = {7},
 pages = {2382--2414},
 year = {2022}
}

@article{Vojta1999,
 author = {Vojta, Paul},
 title = {Integral points on subvarieties of semiabelian varieties. {II}},
 journal = {Amer. J. Math.},
 volume = {121},
 number = {2},
 pages = {283--313},
 year = {1999}
}

@article{Wu,
 author = {Wu, Hung-Hsi},
 title = {Normal families of holomorphic mappings},
 journal = {Acta Math.},
 issn = {0001-5962},
 volume = {119},
 pages = {193--233},
 year = {1967}
}

@article{Yamanoi2019,
 author = {Yamanoi, Katsutoshi},
 title = {Pseudo {Kobayashi} hyperbolicity of subvarieties of general type on abelian varieties},
 journal = {J. Math. Soc. Japan},
 volume = {71},
 number = {1},
 pages = {259--298},
 year = {2019}
}

@article{YamanoiAbelian,
 author = {Yamanoi, Katsutoshi},
 title = {Kobayashi hyperbolicity of the complements of ample divisors in abelian varieties},
 journal = {Publ. Res. Inst. Math. Sci.},
 issn = {0034-5318},
 volume = {58},
 number = {2},
 pages = {281--309},
 year = {2022}
}

@article{Yamanoi2026,
 author = {Yamanoi, Katsutoshi},
 title = {{Bloch}'s principle for holomorphic maps into subvarieties of semi-abelian varieties},
 journal = {Publ. Res. Inst. Math. Sci.},
 volume = {62},
 number = {3},
 pages = {455--628},
 year = {2026},
 doi = {10.4171/PRIMS/62-3-1}
}

\end{document}